\documentclass{article}
\usepackage{amsmath, amsfonts, amsthm, amssymb,amscd, xypic}
\usepackage{graphicx}
\usepackage{float}
\usepackage{verbatim}

\numberwithin{equation}{section}
\newtheorem{thm}{Theorem}[section]
\newtheorem{lma}[thm]{Lemma}

\newtheorem{defn}[thm]{Definition}

\newtheorem{prop}[thm]{Proposition}

\newtheorem{rem}[thm]{Remark}

\renewcommand{\geq}{\geqslant}
\renewcommand{\leq}{\leqslant}
\renewcommand{\H}{\text{H}}
\renewcommand{\P}{\text{P}}

\title{Dimension and measure for typical random fractals}

\author{Jonathan M. Fraser\\ \\
\emph{Mathematical Institute, University of St Andrews, North Haugh,}\\ \emph{St Andrews, Fife, KY16 9SS, Scotland}\\ \emph{e-mail: jmf32@st-andrews.ac.uk}}

\begin{document}
\maketitle

\begin{abstract}
A random iterated function system (RIFS) is a finite set of (deterministic) iterated function systems (IFSs) acting on the same metric space and, for a given RIFS, we define a continuum of random attractors corresponding to each sequence of deterministic IFSs.  Much work has been done on computing the `almost sure' dimensions of these random attractors.  We compute the typical dimensions (in the sense of Baire) and observe that our results are in stark contrast to those obtained using the probabilistic approach.  Furthermore, we examine the typical Hausdorff and packing measures of the random attractors and give examples to illustrate some of the strange phenomena that can occur.  The only restriction we impose on the maps is that they are bi-Lipschitz and we obtain our dimension results without assuming any separation conditions.
\end{abstract}

\section{Introduction}

In this paper we consider the dimension and measure of typical attractors of random iterated function systems (RIFSs).  We define a RIFS to be a finite set of iterated function systems (IFSs) acting on the same metric space and, for a given RIFS, we define a continuum of random attractors corresponding to each sequence of deterministic IFSs.  In fact, these attractors are 1-variable random fractals, as discussed in \cite{superfractals, vvariable2, vvariable}.  Much work has been done on computing the `almost sure' dimensions of these random attractors, where `almost sure' refers to a probability measure on the sample space induced from a probability vector associated with the finite list of IFSs.  One expects the dimension to be some sort of `weighted average' of the dimensions corresponding to the attractors of the deterministic IFSs.  Here, we consider a \emph{topological} approach, based on Baire category, to computing the generic dimensions and obtain results in stark contrast to those obtained using the probabilistic approach.  We are able to obtain very general results, only requiring that our maps are bi-Lipschitz and assuming no separation conditions.  We compute the typical Hausdorff, packing and box dimensions of the random attractors (in the sense of Baire) and also study the typical Hausdorff and packing measures with respect to different gauge functions.  Finally, we give a number of illustrative examples and open questions.
\\ \\
We find that the \emph{dimensions} of typical attractors behave rather well.  In particular, the typical Hausdorff and lower box dimension are always as small as possible and the typical packing and upper box dimensions are always as large as possible.  In comparison, the typical Hausdorff and packing \emph{measures} behave rather badly. We provide examples where the typical Hausdorff measure in the critical dimension is as small as possible and examples where it is as large as possible (with similar examples concerning packing measure).  We find that in the simpler setting of random self-similar sets the behaviour of the typical Hausdorff and packing measures is more predictable.

\subsection{The random model} \label{randommodel}

Let $(K, d)$ be a compact metric space.  A (deterministic) iterated function system (IFS) is a finite set of contraction mappings on $K$.  Given such an IFS, $\{S_1, \dots, S_m\}$, it is well-known that there exists a unique non-empty compact set $F$ satisfying
\[
F= \bigcup_{i=1}^{m} S_{i}(F)
\]
which is called the attractor of the IFS.  We define a random iterated function system (RIFS) to be a set $\mathbb{I} = \{ \mathbb{I}_1, \dots, \mathbb{I}_N\}$, where each $\mathbb{I}_i$ is a deterministic IFS, $\mathbb{I}_i = \{ S_{i,j} \}_{j \in \mathcal{I}_i}$, for a finite index set, $\mathcal{I}_i$, and each map, $S_{i,j}$, is a contracting bi-Lipschitz self-map on $K$.  We define a continuum of attractors of $\mathbb{I}$ in the following way.  Let $D= \{1, \dots, N\}$, $\Omega = D^\mathbb{N}$ and let $\omega = (\omega_1, \omega_2, \dots) \in \Omega$.  Define the attractor of $\mathbb{I}$ corresponding to $\omega$ by
\[
F_\omega \, = \, \bigcap_{k}\  \bigcup_{i_1\in \mathcal{I}_{\omega_1}, \dots, i_k \in\mathcal{I}_{\omega_k}} S_{\omega_1, i_1} \circ \dots \circ S_{\omega_k,i_k}(K).
\]
So, by `randomly choosing' $\omega \in \Omega$, we `randomly choose' an attractor $F_\omega$.  Attractors of RIFSs can enjoy a much richer and more complicated structure than attractors of IFSs.  Some pictures have been included in Section \ref{pics} to help illustrate this.  We now wish to make statements about the generic nature of $F_\omega$.  In particular, what is the generic dimension of $F_\omega$?  In the following section we briefly recall the notions of dimension we will be interested in.

\subsection{Dimension and measure}

Let $F$ be a subset of $K$.  For $s \geq 0$ and $\delta>0$ we define the $\delta$-approximate $s$-dimensional Hausdorff measure of $F$ by

\[
\mathcal{H}_\delta^s (F) = \inf \bigg\{ \sum_{i \in I} \lvert U_i \rvert^s : \{ U_i \}_{i \in I} \text{ is a countable $\delta$-cover of $F$ by open sets} \bigg\}
\]

and the $s$-dimensional Hausdorff (outer) measure of $F$ by $\mathcal{H}^s (F) = \lim_{\delta \to 0} \mathcal{H}_\delta^s (F)$.  The Hausdorff dimension of $F$ is
\[
\dim_\text{H} F = \inf \Big\{ s \geq 0: \mathcal{H}^s (F) =0 \Big\} = \sup \Big\{ s \geq 0: \mathcal{H}^s (F) = \infty \Big\}.
\]

If $F$ is compact, then we may define the Hausdorff measure of $F$ in terms of \emph{finite} covers.  Packing measure, defined in terms of \emph{packings}, is a natural dual to Hausdorff measure, which was defined in terms of \emph{covers}.  For $s \geq 0$ and $\delta>0$ we define the $\delta$-approximate $s$-dimensional packing pre-measure of $F$ by

\[
\mathcal{P}_{0,\delta}^s (F) = \sup \bigg\{ \sum_{i \in I} \lvert U_i \rvert^s : \{ U_i \}_{i \in I} \text{ is a countable centered $\delta$-packing of $F$ by closed balls} \bigg\}
\]

and the $s$-dimensional packing pre-measure of $F$ by $\mathcal{P}_0^s (F) = \lim_{\delta \to 0} \mathcal{P}_{0,\delta}^s (F)$.  To ensure countable stability, the packing (outer) measure of $F$ is defined by

\[
\mathcal{P}^s (F) = \inf \bigg\{\sum_i \mathcal{P}_0^s (F_i) : F \subseteq \bigcup_i F_i  \bigg\}
\]

and the packing dimension of $F$ is

\[
\dim_\text{P} F = \inf \Big\{ s \geq 0: \mathcal{P}^s (F) =0 \Big\} = \sup \Big\{ s \geq 0: \mathcal{P}^s (F) = \infty \Big\}.
\]

A less sophisticated, but very useful, notion of dimension is box (or box-counting) dimension.  The lower and upper box dimensions of $F$ are defined by
\[
\underline{\dim}_\text{B} F = \liminf_{\delta \to 0} \, \frac{\log N_\delta (F)}{-\log \delta}
\qquad
\text{and}
\qquad
\overline{\dim}_\text{B} F = \limsup_{\delta \to 0} \,  \frac{\log N_\delta (F)}{-\log \delta},
\]
respectively, where $N_\delta (F)$ is smallest number of open sets required for a $\delta$-cover of $F$.  If $\underline{\dim}_\text{B} F = \overline{\dim}_\text{B} F$, then we call the common value the box dimension of $F$ and denote it by $\dim_\text{B} F$.
\\ \\
In general we have the following relationships between the dimensions discussed above
\[
\begin{array}{ccccc}
& &                                                                              \dim_\text{\P} F                                    & &  \\
 &                                 \rotatebox[origin=c]{45}{$\leq$}          & &                \rotatebox[origin=c]{315}{$\leq$} &  \\
 \dim_\text{\H} F                                             & & & &                         \overline{\dim}_\text{B} F \\
 &                                 \rotatebox[origin=c]{315}{$\leq$}       & &                  \rotatebox[origin=c]{45}{$\leq$} &  \\
& &                                                                        \underline{\dim}_\text{B} F                           & &  
\end{array}
\]
and, furthermore, if $F$ is compact and every ball centered in $F$ intersects $F$ in a set with upper box dimension the same as $F$, then $\dim_\text{\H} F  \leq  \underline{\dim}_\text{B} F  \leq \dim_\text{\P} F = \overline{\dim}_\text{B} F$.  This will apply in our situation.
\\ \\
It is possible to consider a `finer' definition of Hausdorff and packing dimension.  We define a \emph{gauge function} to be a function, $G: (0,\infty) \to (0,\infty)$, which is continuous, monotonically increasing and satisfies $\lim_{t \to 0} G(t) = 0$.  We then define the Hausdorff measure, packing pre-measure and packing measure with respect to the gauge $G$ as
\[
\mathcal{H}^G (F) = \lim_{\delta \to 0} \, \inf \bigg\{ \sum_{i \in I} G(\lvert U_i \rvert) : \{ U_i \}_{i \in I} \text{ is a countable $\delta$-cover of $F$ by open sets} \bigg\},
\]
\[
\mathcal{P}^G_0 (F) = \lim_{\delta \to 0} \, \sup \bigg\{ \sum_{i \in I} G( \lvert U_i \rvert) : \{ U_i \}_{i \in I} \text{ is a countable centered $\delta$-packing of $F$ by closed balls} \bigg\}
\]
and
\[
\mathcal{P}^G (F) = \inf \bigg\{\sum_i \mathcal{P}_0^G (F_i) : F \subseteq \bigcup_i F_i  \bigg\}
\]
respectively.  Note that if $G(t) = t^s$ then we obtain the standard Hausdorff and packing measures.  The advantage of this approach is that, in the case where the measure of a set is zero or infinite in its dimension, one may be able to find an appropriate gauge for which the measure is positive and finite.  For example, with probability 1, Brownian trails in $\mathbb{R}^2$ have Hausdorff dimension 2, but 2-dimensional Hausdorff measure equal to zero.  However, with probability 1, they have positive and finite $\mathcal{H}^G$-measure with respect to the gauge $G(t) = t^2 \log(1/t) \, \log \log \log(1/t) $, see \cite{falconer}, Chapter 16, and the references therein.
\\ \\
For a given gauge function, $G$, and a constant, $c>0$, we define
\[
D^-(G,c) = \inf_{t>0} \frac{G(c \, t)}{G(t)} \qquad \text{and} \qquad D^+(G,c) = \sup_{t>0} \frac{G(c \, t)}{G(t)}.
\]
Notice that, if $c\leq 1$, then $D^+(G,c) \leq 1$.  It is easy to see that if $0<D^-(G,c) \leq D^+(G,c) < \infty$ for some $c>0$, then $0<D^-(G,c) \leq D^+(G,c) < \infty$ for all $c>0$ and in this case we say that the gauge is \emph{doubling}.  The standard gauge is clearly doubling, with $D^-(G,c) = D^+(G,c) = c^s$.
\\ \\
For a more detailed discussion of this finer approach to dimension see, \cite{falconer}, Section 2.5, or \cite{rogers}, Chapter 2.

\begin{rem}
We have defined Hausdorff measure and box dimension by means of covers by open sets.  We do this for technical reasons and note that these definitions are equivalent to the standard definitions using  covers by arbitrary sets, see \cite{mattila} Theorem 4.4.
\end{rem}

\subsection{Separation properties}

In this section we introduce some separation properties which will be required for some of our results.  Note that our main Theorem (Theorem \ref{main}) does not require any separation properties.

\begin{defn}
We say that a deterministic IFS, $\{S_i\}_{i=1}^m$, satisfies the \emph{open set condition (OSC)}, if there exists a non-empty open set, $\mathcal{O}$, such that
\[
\bigcup_{i=1}^m S_{i}(\mathcal{O}) \subseteq \mathcal{O}
\]
with the union disjoint.
\end{defn}

We generalise the OSC to the RIFS situation in the following way.

\begin{defn}
We say that $\mathbb{I}$ satisfies the \emph{uniform open set condition (UOSC)}, if each deterministic IFS satisfies the OSC and the open set can be chosen uniformly, i.e., there exists a non-empty open set $\mathcal{O} \subseteq K$ such that, for each $i \in D$, we have
\[
\bigcup_{j \in \mathcal{I}_i} S_{i,j}(\mathcal{O}) \subseteq \mathcal{O}
\]
with the union disjoint.
\end{defn}

The UOSC also appears in, for example, \cite{vvariable2}.

\begin{defn}
Let $\mu$ be a Borel measure supported on $K$.  We say that $\mathbb{I}$ satisfies the \emph{$\mu$-measure separated condition ($\mu$-MSC)}, if, for all $\omega \in \Omega, \ l \in D$ and $i, j \in \mathcal{I}_{l}$ with $i \neq j$, we have
\[
\mu\big(S_{l, i}(F_\omega) \cap S_{l, j}(F_\omega) \big) = 0.
\]
\end{defn}

The $\mu$-MSC means that $\mu$ will be additive on the subsets of $F_\omega$ corresponding to images of finite (distinct) sequences of maps, $S_{\omega_1, i_1}, \dots, S_{\omega_k, i_k}$.  In this paper we will use the $\mu$-MSC with $\mu$ equal to either the Hausdorff or packing measure.

\subsection{Bi-Lipschitz maps}

In Section 1.1 we mentioned that our maps are assumed to be bi-Lipschitz contractions.  In this section we will fix some related notation which we will need to state some of our results.  For a map $\phi:K \to K$ define
\[
\text{Lip}^-(\phi) = \inf_{x,y \in K} \,  \frac{d\big(\phi(x), \phi(y)\big)}{d(x,y)} \qquad \text{and} \qquad \text{Lip}^+(\phi) = \sup_{x,y \in K} \, \frac{d\big(\phi(x), \phi(y)\big)}{d(x,y)}.
\]
If $\text{Lip}^+(\phi)<\infty$, then we say $\phi$ is \emph{Lipschitz} and if, in addition, $\text{Lip}^-(\phi)>0$, then we say $\phi$ is \emph{bi-Lipschitz}.  If $\text{Lip}^+(\phi)<1$, then we say $\phi$ is a \emph{contraction}.  Finally, if $\text{Lip}^-(\phi) = \text{Lip}^+(\phi)< 1$, then we write $\text{Lip}(\phi)$ to denote the common value and say that $\phi$ is a \emph{similarity}.  Given a deterministic IFS, $\{ S_{1}, \dots,  S_{k}\}$, consisting of similarities, the \emph{similarity dimension} is defined to be the unique solution to Hutchison's formula
\[
\sum_{i=1}^{k} \text{Lip}(S_i)^s = 1.
\]
It is well-known that if such an IFS satisfies the OSC, then the similarity dimension equals the Hausdorff, packing and box dimension of the attractor, see \cite{falconer} Section 9.3.

\subsection{The probabilistic approach} \label{probapp}

The most common approach to studying random fractals is to associate a probability measure with the space of possible attractors and then make \emph{almost sure} statements.  For some examples based on conformal systems, see \cite{randomfractals}, \cite{randomselfconformal}, \cite{olsenbook}, \cite{vvariable2}, \cite{vvariable}, \cite{superfractals}; and for non-conformal (self-affine) systems, see \cite{random_carpets}, \cite{random_carpets2} \cite{statselfaffine}, \cite{randomsponges}, \cite{me_random}.  For the random model we described in Section \ref{randommodel} this probabilistic approach would go as follows.  Associate a probability vector, $\textbf{\emph{p}} = (p_1, \dots, p_N)$, with $\mathbb{I}$.  Then, to obtain our random attractor, we choose each entry in $\omega$ randomly and independently with respect to $\textbf{\emph{p}}$.  This induces a probability measure, $\mathbb{P}$, on $\Omega$ given by
\[
\mathbb{P} = \prod_\mathbb{N}  \, \sum_{i=1}^{N} p_i \, \delta_i,
\]
where $\delta_i$ is the Dirac measure concentrated at $i \in D= \{1, \dots, N\}$.  We then say a property of the random attractors is generic if it occurs for $\mathbb{P}$-almost all $\omega \in \Omega$.  This approach has attracted much attention in the literature with the ergodic theorem often playing a key role in the analysis, utilising the fact that $\mathbb{P}$ is ergodic with respect to the left shift on $\Omega$.  We give a couple of examples.

\begin{thm}[\cite{vvariable2}] \label{almostsuress}
Let $\mathbb{I}=\{\mathbb{I}_1, \dots, \mathbb{I}_N\}$ be an RIFS consisting of similarity maps on $\mathbb{R}^n$ with associated probability vector $\textbf{p}=(p_1, \dots, p_N)$.  Assume that $\mathbb{I}$ satisfies the UOSC and let $s$ be the solution of
\begin{equation} \label{randomhutch}
\prod_{i=1}^{N}  \bigg(\sum_{j \in \mathcal{I}_i} \text{\emph{Lip}}(S_{i,j})^s\bigg)^{p_i} = 1.
\end{equation}
Then, for $\mathbb{P}$ almost all $\omega \in \Omega$, $\dim_\text{\emph{H}} F_\omega  = \dim_\text{\emph{B}} F_\omega  = \dim_\text{\emph{P}} F_\omega  = s$.
\end{thm}

Equation \ref{randomhutch} should be viewed as a randomised version of Hutchison's formula.  Here the almost sure dimension is `some sort of' weighted average of the dimensions of the attractors of $\mathbb{I}_i$.  For a proof of Theorem \ref{almostsuress}, see \cite{vvariable2}, or alternatively, \cite{superfractals}, Chapter 5.7, and the references therein.
\\ \\
Self-affine sets are an important class of fractals and often provide examples of strange behaviour not observed in the self-similar setting.  We will now discuss a well-studied class of self-affine sets and random self-affine sets which we will use in Section \ref{examples} to demonstrate some important phenomena.  Take the unit square and divide it up into an $m \times n$ grid for some $m\leq n$.  Now choose a subset of the rectangles formed by the grid and form an IFS of affine maps which take the unit square onto each chosen subrectangle, preserving orientation.  The attractor of this system is called a self-affine Sierpi\'nski carpet.  A formula for the Hausdorff dimension was obtained independently by Bedford \cite{bedford} and McMullen \cite{mcmullen}.  Now consider a random Sierpi\'nski carpet where we take $N$ deterministic IFSs, $\mathbb{I}_i$, built by dividing the unit square into an $m_i \times n_i$ grid $m_i \leq n_i$ and an associated probability vector $\textbf{\emph{p}}=(p_1, \dots, p_N)$.  The following dimension formula was given in \cite{random_carpets2} and can be derived from results in \cite{me_random}.

\begin{thm}[\cite{me_random}, \cite{random_carpets2}] \label{melars}
For $j = 1 \dots m_i$, let $C_{i,j} \in \{0,\dots, m_i\}$ denote the number of rectangles chosen in the $j$th column in the $i$th IFS.  Let
\[
\nu_1 = m_1^{p_1} \cdots m_N^{p_N} \qquad \text{and} \qquad \nu_2 = n_1^{p_1} \cdots n_N^{p_N}.
\]
Then, for $\mathbb{P}$ almost all $\omega \in \Omega$,
\[
\dim_\text{\emph{H}} F_\omega  = \sum_{i=1}^N p_i \,  \Bigg(\frac{1}{\log \nu_1} \log\bigg(\sum_{j=1}^{m_i} C_{i,j}^{\log \nu_1 / \log \nu_2}\bigg) \Bigg).
\]
\end{thm}

We note that in \cite{me_random} a higher dimensional analogue of Theorem \ref{melars} was obtained where one begins the construction with the unit cube in $\mathbb{R}^d$ rather than the unit square.  Notice that if $m_i=m$ and $n_i=n$ for all $i$, then the above dimension formula simplifies to
\[
\dim_\H F_\omega  = \sum_{i=1}^N p_i \,  \Bigg(\frac{1}{\log m} \log\bigg(\sum_{j=1}^{m} C_{i,j}^{\log m / \log n}\bigg) \Bigg) = \sum_{i=1}^N p_i \, s_i,
\]
where $s_i$ is the Hausdorff dimension of the attractor of the attractor of $\mathbb{I}_i$ given by Bedford and McMullen.  In this case, the almost sure Hausdorff and box dimension were computed in \cite{random_carpets}.  If the $m_i$ and $n_i$ are not chosen uniformly, then we have a nonlinear dependence on the probability vector $\textbf{\emph{p}}$.  An example using Theorem \ref{melars} will be given in Section \ref{affineexample}.

\subsection{The topological approach} \label{topapp}

In this paper we will investigate the generic dimension and measure of $F_\omega$ from a topological point of view using Baire Category.  In this section we will recall the basic definitions and theorems.
\\ \\
Let $(X,d)$ be a complete metric space.  A set $N \subseteq X$ is \emph{nowhere dense} if for all $x \in N$ and for all $r>0$ there exists a point $y \in X \setminus N$ and $t>0$ such that
\[
B(y,t) \subseteq B(x,r) \setminus N.
\]
A set $M$ is said to be \emph{of the first category}, or, \emph{meagre}, if it can be written as a countable union of nowhere dense sets.  We think of a meagre set as being \emph{small} and the complement of a meagre set as being \emph{big}.  A set $T \subseteq X$ is \emph{residual} or \emph{co-meagre}, if $X \setminus T$ is meagre.  A property is called \emph{typical} if the set of points which have the property is residual.  In Section \ref{proofs} we will use the following theorem to test for typicality without mentioning it explicitly.

\begin{thm}
In a complete metric space, a set $T$ is residual if and only if $T$ contains a countable intersection of open dense sets or, equivalently, $T$ contains a dense $G_\delta$ subset of $X$.
\end{thm}

\begin{proof}  See \cite{oxtoby}. \end{proof}

In order to consider typical properties of members of $\Omega$, we need to topologize $\Omega$ in a suitable way.  We do this by equipping it with the metric $d_\Omega$ where, for $u = (u_1, u_2, \dots) \neq v = (v_1, v_2, \dots) \in \omega$,
\[
d_\Omega(u,v) = 2^{-k}
\]
where $k= \min\{n \in \mathbb{N} : u_n \neq v_n\}$.  The space $(\Omega, d_\Omega)$ is complete.  For a more detailed account of Baire Category the reader is referred to \cite{oxtoby}.
\\ \\
It is worth noting that one could also formulate the topological approach using the set $\{F_\omega : \omega \in \Omega\}$ instead of $\Omega$.  In fact, this leads to an equivalent analysis but since we do not use this approach directly we defer discussion of it until Section \ref{discuss} (9).

\section{Results} \label{results}

In this section we state our results.  In Section \ref{general} we state results which apply in very general circumstances, namely, the random iterated function systems introduced in Section 1.1.  Theorem \ref{main} is the main result of the paper and gives the typical Hausdorff, packing and upper and lower box dimensions of $F_\omega$ and, furthermore, gives sufficient conditions for the typical Hausdorff and packing measures with respect to any (doubling) gauge function to be zero or infinite.  In Section \ref{ssresults} we specialise to the self-similar setting.

\subsection{Results in the general setting} \label{general}

Our main result is the following.

\begin{thm} \label{main}
Let $G: (0,\infty) \to (0,\infty)$ be a gauge function.
\begin{itemize}
\item[(1)] If $\inf_{u \in \Omega} \,\mathcal{H}^G( F_u) = 0$, then for a typical $\omega \in \Omega$, we have $\mathcal{H}^G(F_\omega)  = 0$;
\item[(2)] If $G$ is doubling and $\sup_{u \in \Omega} \,\mathcal{P}^G( F_u) = \infty$, then for a typical $\omega \in \Omega$, we have $\mathcal{P}^G(F_\omega)  = \infty$;
\item[(3)] The typical Hausdorff dimension is infimal, i.e., for a typical $\omega \in \Omega$, we have
\[
\dim_\text{\emph{H}} F_\omega = \inf_{u \in \Omega} \, \dim_\text{\emph{H}} F_u;
\]
\item[(4)] The packing dimension and upper box dimension are supremal and, in fact, for a typical $\omega \in \Omega$, we have
\[
\overline{\dim}_\text{\emph{B}} F_\omega = \dim_\text{\emph{P}} F_\omega = \sup_{u \in \Omega} \, \overline{\dim}_\text{\emph{B}} F_u= \sup_{u \in \Omega} \, \dim_\text{\emph{P}} F_u;
\]
\item[(5)] The lower box dimension is infimal, i.e, for a typical $\omega \in \Omega$, we have
\[
\underline{\dim}_\text{\emph{B}} F_\omega = \inf_{u \in \Omega} \, \underline{\dim}_\text{\emph{B}} F_u.
\]
\end{itemize}
\end{thm}

We will prove Theorem \ref{main} part (1) in Section \ref{hmeasure}; part (2) in Section \ref{pmeasure}; and part (5) in Section \ref{lbox}.  Choosing $G$ such that $G(t) = t^s$, part (3) follows from part (1) and part (4) follows from part (2) combined with the observation that the packing and upper box dimension coincide for all random attractors, see Lemma \ref{packingupperbox}.
\\ \\
It is slightly unsatisfactory that in Theorem \ref{main} part (1) we do not get a precise value for the typical Hausdorff measure if the infimal Hausdorff measure is positive and finite; and similarly,  in part (2) we do not get a precise value for the typical packing measure if the supremal packing measure is positive and finite.  In keeping with the rest of the results and what is `usually' expected when dealing with Baire category, one might expect that either: the typical Hausdorff measure will be the infimal value and the typical packing measure will be the supremal value; or, even though $F_\omega$ will typically be `small' in terms of Hausdorff dimension and `large' in terms of packing dimension, due to the influence of deterministic IFSs with non-extremal attractors, they will be `large' in terms of Hausdorff measure and `small' in terms of packing measure.  Surprisingly, both of these phenomena are possible.  In the following two theorems we identify a large class of RIFS where the second type of behaviour occurs.  Theorem \ref{infinite} refers to Hausdorff measure and Theorem \ref{zero} refers to packing measure.

\begin{thm} \label{infinite}
Write $h= \inf_{u \in \Omega} \, \dim_\text{\emph{H}} F_u$ and assume that $\mathbb{I}$ satisfies the $\mathcal{H}^h$-MSC and that there exists $v=(v_1, v_2, \dots) \in \Omega$ such that
\begin{equation} \label{lip-}
\lim_{l \to \infty} \sum_{j_1\in \mathcal{I}_{v_1}, \dots, j_l \in\mathcal{I}_{v_l}} \text{\emph{Lip}}^{-}(S_{v_1, j_1} \circ \dots \circ S_{v_l,j_l})^h = \infty.
\end{equation}
Then, 
\begin{itemize}
\item[(1)] If $\inf_{u \in \Omega} \,\mathcal{H}^h( F_u) = 0$, then for a typical $\omega \in \Omega$, we have $\mathcal{H}^h(F_\omega) = 0$;
 \item[(2)] If $\inf_{u \in \Omega} \,\mathcal{H}^h( F_u) >0$, then for a typical $\omega \in \Omega$, we have $\mathcal{H}^h(F_\omega) = \infty$.
\end{itemize}
\end{thm}

Note that part (1) follows from Theorem \ref{main}.  We will prove Theorem \ref{infinite} (2) in Section \ref{infiniteproof}.  Although condition (\ref{lip-}) seems a little contrived, what it really means is that, for some $v \in \Omega$, we can give a simple lower bound for the Hausdorff dimension of $F_v$ which is strictly bigger than the infimal Hausdorff dimension, $h$.  

\begin{thm} \label{zero}
Write $p= \sup_{u \in \Omega} \, \dim_\text{\emph{P}} F_u$ and assume that there exists $v=(v_1, v_2, \dots) \in \Omega$ such that
\begin{equation} \label{lip+}
\lim_{k \to \infty} \sum_{j_1\in \mathcal{I}_{v_1}, \dots, j_k \in\mathcal{I}_{v_k}} \text{\emph{Lip}}^{+}(S_{v_1, j_1} \circ \dots \circ S_{v_k,j_k})^p = 0.
\end{equation}
Then, 
\begin{itemize}
\item[(1)] If $\sup_{u \in \Omega} \,\mathcal{P}^p( F_u) = \infty$, then for a typical $\omega \in \Omega$, we have $\mathcal{P}^p(F_\omega) = \infty$;
 \item[(2)] If $\sup_{u \in \Omega} \,\mathcal{P}^p( F_u) < \infty$, then for a typical $\omega \in \Omega$, we have $\mathcal{P}^p(F_\omega) = 0$.
\end{itemize}
\end{thm}

Note that in Theorem \ref{zero} we do not require any separation conditions.  Part (1) follows from Theorem \ref{main}.  We will prove Theorem \ref{zero} (2) in Section \ref{zeroproof}.  Similar to above, condition (\ref{lip+}) seems a little contrived at first sight but what it really means is that, for some $v \in \Omega$, we can give a simple upper bound for the packing dimension of $F_v$ which is strictly smaller than the supremal packing dimension, $p$.
\\ \\
With the previous two Theorems in mind, one might be tempted to think that something much more general is true.  Namely, that for $s\geq 0$, we have
\begin{itemize}
 \item[(1)] If $\inf_{u \in \Omega} \,\mathcal{H}^s( F_u) >0$, then for a typical $\omega \in \Omega$, we have $\mathcal{H}^s(F_\omega) = \sup_{u \in \Omega} \,\mathcal{H}^s( F_u) $;
 \item[(2)] If $\sup_{u \in \Omega} \,\mathcal{P}^s( F_u) < \infty$, then for a typical $\omega \in \Omega$, we have $\mathcal{P}^s(F_\omega) =\inf_{u \in \Omega} \,\mathcal{P}^s( F_u)$.
\end{itemize}
However, this is false.  We will demonstrate this by constructing two simple examples in Section \ref{conjfalse}.  This `bad behaviour' of the typical packing and Hausdorff measures disappears to a certain extent if the mappings in the RIFS are similarities.  This idea will be developed in the following section.

\subsection{Results in the self-similar setting} \label{ssresults}

In this section we extend the results of the previous section in the self-similar setting.  It turns out that for random self-similar sets we can obtain more precise information and, furthermore, many of the strange phenomena which we observe in the general setting no longer occur.  The first example of this is that, given the UOSC, the dimensions of $F_\omega$ are bounded by the dimensions of the attractors of the deterministic IFSs.  This allows us to get our hands on the extremal quantities, see Theorem \ref{SS}.  Unfortunately, this rather nice property does not always hold in the general situation.  In Section \ref{affineexample} we will give an example of a RIFS satisfying the UOSC for which the infimal (and thus typical) Hausdorff dimension is strictly less than the minimum Hausdorff dimension of the attractors of the deterministic IFSs.  Secondly, given the UOSC and certain measure separation, we can compute the exact value of the typical Hausdorff and packing measure, see Theorem \ref{SS2}, which we are unable to do in the general situation.
\\ \\
Throughout this section let $\mathbb{I}$ be a RIFS consisting of finitely many deterministic IFSs of similarity mappings of $\mathbb{R}^n$.  For each $i \in D$, let $s_i$ be the solution of
\[
\sum_{j \in \mathcal{I}_i} \text{Lip}(S_{i,j})^{s_i} = 1
\]
and write $s_{\min} = \min_{i \in D} s_i$ and $s_{\max} = \max_{i \in D} s_i$.

\begin{thm} \label{SS}
Assume the UOSC is satisfied.  Then
\begin{itemize} 
\item[(1)] $0< \sup_{\omega \in \Omega} \,  \mathcal{P}^{s_{\max}}(F_\omega) < \infty$;
\item[(2)] $\sup_{\omega \in \Omega} \, \dim_\text{\emph{P}} F_\omega = \sup_{\omega \in \Omega} \, \overline{\dim}_\text{\emph{B}} F_\omega  =  s_{\max}$;
\item[(3)] $0<\inf_{\omega \in \Omega} \,  \mathcal{H}^{s_{\min}}(F_\omega) < \infty$;
\item[(4)] $\inf_{\omega \in \Omega} \, \dim_\text{\emph{H}} F_\omega = \inf_{\omega \in \Omega} \, \underline{\dim}_\text{\emph{B}} F_\omega  = s_{\min}$.
\end{itemize}
\end{thm}

We will prove Theorem \ref{SS} parts (1) and (3) in Section \ref{SSproof}.  Part (2) follows from part (1) and part (4) follows from part (3).  Given certain measure separation we can also compute the exact packing and Hausdorff measure for typical $F_\omega$.  Write $\mathcal{H}_{\min} =\inf_{\omega \in \Omega} \,  \mathcal{H}^{s_{\min}}(F_\omega)$ and $\mathcal{P}_{\max} =  \sup_{\omega \in \Omega} \,  \mathcal{P}^{s_{\max}}(F_\omega) $.

\begin{thm} \label{SS2}
Assume that $\mathbb{I}$ satisfies the UOSC and the $\mathcal{P}^{s_{\min}}$-MSC.  Then
\begin{itemize} 
\item[(1)] If $s_{\min} = s_{\max} = s$, then for a typical $\omega \in \Omega$,
\[
\dim_\text{\emph{H}} F_\omega = \dim_\text{\emph{P}} F_\omega = s
\]
and
\[
0 < \mathcal{H}^s(F_\omega) = \mathcal{H}_{\min} \leq \mathcal{P}_{\max}  = \mathcal{P}^s(F_\omega) < \infty;
\]
\item[(2)] If $s_{\min} < s_{\max}$, then for a typical $\omega \in \Omega$,
\[
\dim_\text{\emph{H}} F_\omega = s_{\min} < s_{\max} = \dim_\text{\emph{P}} F_\omega,
\]
\[
\mathcal{H}^{s_{\min}}(F_\omega) = \infty
\]
and
\[
\mathcal{P}^{s_{\max}}(F_\omega) = 0.
\]
\end{itemize}
\end{thm}
We will prove Theorem \ref{SS2} (1) in Section \ref{SSproof2}.  Note that part (2) follows immediately from Theorems \ref{infinite} and \ref{zero}.  In Section \ref{notdense} we construct a simple example where we can apply Theorem \ref{SS2}.
\\ \\
It is worth noting here that it is possible to give easily checkable sufficient conditions for the $\mathcal{P}^{s_{\min}}$-MSC to hold.  In particular, if we say that $\mathbb{I}$ satisfies the uniform strong open set condition (USOSC) if the the UOSC is satisfied and the open set $\mathcal{O}$ can be chosen such that, for every $\omega \in \Omega$, we have $\mathcal{O} \cap F_\omega \neq \emptyset$, then we can use an argument similar to that used by Lalley in \cite{lalley}, Section 6, to show that the $\mathcal{P}^{s_{\min}}$-MSC is satisfied.  Unfortunately, the USOSC is not equivalent to the UOSC as in the deterministic case, see \cite{schief}.
\\ \\
We can also obtain a partial result concerning packing measure without assuming any separation conditions.
\begin{thm} \label{corzero}
Each deterministic IFS, $\mathbb{I}_i \in \mathbb{I}$, has an attractor with dimension $d_i$ and similarity dimension $s_i \geq d_i$.  Assume that $s_{\min}< \max_i d_i$.  Write $p= \sup_{u \in \Omega} \, \dim_\text{\emph{P}} F_u$.  Then, for a typical $\omega \in \Omega$, $\dim_\text{\emph{P}} F_\omega = p$, but $\mathcal{P}^p(F_\omega) = 0$.
\end{thm}

\begin{proof}
This follows immediately from Theorem \ref{zero}.
\end{proof}

\newpage

\section{Proofs} \label{proofs}

Throughout this section let $G: (0,\infty) \to (0,\infty)$ be a gauge function.

\subsection{Preliminary observations}

In this section we will gather together some simple preliminary results and observations which will be used in the subsequent sections without being mentioned explicitly.  The proofs are elementary (or classical) and are omitted.

\begin{lma}[scaling properties] \label{scaling}

Let $\phi: K \to K$ be a bi-Lipschitz map and $F \subseteq K$.  Then
\[
D^-(G, \text{\emph{Lip}}^-(\phi)) \, \mathcal{H}^G(F) \leq \mathcal{H}^G(\phi(F)) \leq D^+(G, \text{\emph{Lip}}^+(\phi))\,  \mathcal{H}^G(F), 
\]
\[
D^-(G, \text{\emph{Lip}}^-(\phi)) \, \mathcal{P}^G_0(F) \leq \mathcal{P}^G_0(\phi(F)) \leq D^+(G, \text{\emph{Lip}}^+(\phi))\,  \mathcal{P}^G_0(F)
\]
and
\[
D^-(G, \text{\emph{Lip}}^-(\phi)) \, \mathcal{P}^G(F) \leq \mathcal{P}^G(\phi(F)) \leq D^+(G, \text{\emph{Lip}}^+(\phi))\,  \mathcal{P}^G(F).
\]
In particular, using the standard gauge,
\[
\text{\emph{Lip}}^-(\phi)^s \, \mathcal{H}^s(F) \leq \mathcal{H}^s(\phi(F)) \leq  \text{\emph{Lip}}^+(\phi)^s \,  \mathcal{H}^s(F),
\]
\[
\text{\emph{Lip}}^-(\phi)^s \, \mathcal{P}^s_0(F) \leq \mathcal{P}^s_0(\phi(F)) \leq  \text{\emph{Lip}}^+(\phi)^s \,  \mathcal{P}^s_0(F)
\]
and
\[
\text{\emph{Lip}}^-(\phi)^s \, \mathcal{P}^s(F) \leq \mathcal{P}^s(\phi(F)) \leq  \text{\emph{Lip}}^+(\phi)^s \,  \mathcal{P}^s(F).
\]
\end{lma}
Lemma \ref{scaling}, says that if the gauge is doubling, then mapping a set under a bi-Lipschitz map only changes the measure by a constant.  Clearly if $\phi$ is bi-Lipschitz, then $\dim \phi(F) = \dim F$, where $\dim$ can be any of the four dimensions used here.  We can also deduce that, for all $\omega \in \Omega$, the upper box dimension and packing dimension coincide.

\begin{lma}[packing and upper box dimension] \label{packingupperbox}
For all $\omega \in \Omega$, $\dim_\text{\emph{P}} F_\omega = \overline{\dim}_\text{\emph{B}} F_\omega$.
\end{lma}

To prove this simply note that all balls centered in $F_\omega$ contain a bi-Lipschitz image of $F_{(\omega_k, \omega_{k+1}, \dots)}$ for some sufficiently large $k$ and, furthermore, $F_\omega$ can be written as a finite union of bi-Lipschitz images of $F_{(\omega_k, \omega_{k+1}, \dots)}$ and since upper box dimension is finitely stable, $\overline{\dim}_\text{B} F_{(\omega_k, \omega_{k+1}, \dots)} = \overline{\dim}_\text{B} F_{\omega}$ and the result follows.  See the discussion on sufficient conditions for the equality of packing and upper box dimension given in Section 1.2.
\\ \\
We recall the defintion of the Hausdorff metric.  Let $\mathcal{K}(K)$ denote the set of all compact subsets of $(K,d)$.  This forms a complete metric space when equipped with the Hausdorff metric, $d_\mathcal{H}$, which is defined by
\[
d_\mathcal{H}(E,F) = \inf \{ \varepsilon>0: E \subseteq F_\varepsilon \text{ and } F \subseteq E_\varepsilon\}
\]
for $E,F \in \mathcal{K}(K)$ and where $E_\varepsilon$ denotes the $\varepsilon$-neighbourhood of $E$.  The following lemma will allow us to approximate $F_\omega$ in $K$ by approximating $\omega$ in $\Omega$, which will be of vital importance in  the subsequent proofs.

\begin{lma}[continuity properties] \label{continuity}
The map $\Psi: \big(\Omega, d_\Omega \big) \to \big(\mathcal{K}(K), d_\mathcal{H}\big)$ defined by $\Psi(\omega) = F_\omega$ is continuous.
\end{lma}

Finally, we will state a version of the mass distribution principle which we use to estimate the Hausdorff and packing measures of random self-similar sets in Section \ref{SSproof}.

\begin{prop}[mass distribution principle] \label{MDP}
Let $\mu$ be a Borel probability measure supported on a Borel set $F \subset \mathbb{R}^n$ and let $\lambda \in (0, \infty)$.  Then
\begin{itemize}
\item[(1)] If $\limsup_{r \to 0} \mu\big(B(x,r)\big) \, r^{-s} \leq \lambda$ for all $x \in F$, then $\mathcal{H}^s(F) \geq \lambda^{-1}$;
\item[(2)] If $\liminf_{r \to 0} \mu\big(B(x,r)\big) \, r^{-s} \geq \lambda$ for all $x \in F$, then $\mathcal{P}^s(F) \leq \lambda^{-1}\, 2^s$.
\end{itemize}
\end{prop}

For a proof of this, see \cite{techniques, mattila}.

\subsection{Proof of Theorem \ref{main} (1)} \label{hmeasure}

Suppose $\inf_{u \in \Omega} \, \mathcal{H}^G(F_u) = 0$.  We will show that the set
\[
H = \{ \omega \in \Omega : \mathcal{H}^G( F_\omega) =0\}
\]
is residual.  Writing $H_{m,n} = \{ \omega \in \Omega : \mathcal{H}^{G}_{1/m} (F_\omega) < \tfrac{1}{n}\}$, we have
\begin{eqnarray*}
H= \bigcap_{m, n \in \mathbb{N}} H_{m,n},
\end{eqnarray*}
so it suffices to prove that each $H_{m,n}$ is open and dense in $(\Omega, d_\Omega)$.  Fix $m, n \in \mathbb{N}$.
\\ \\
(i) $H_{m,n}$ is open.
\\ \\
Let $\omega \in H_{m,n}$.  It follows that there exists a finite $(1/m)$-cover of $F_\omega$ by open sets, $\{U_i\}$,  satisfying
\[
\sum_i G(\lvert U_i \rvert) < \tfrac{1}{n}.
\]
Let $\mathcal{U} = \partial  \big(\cup_i U_i\big)$ be the boundary of the union of the covering sets, $\{U_i\}$, and let
\[
\eta = \min_{x \in \mathcal{U}, y \in F_\omega} d(x,y)
\]
which is strictly positive by the compactness of $F_\omega$.  Now choose $r>0$ sufficiently small to ensure that if $u \in B(\omega, r)$, then $d_\mathcal{H} (F_\omega, F_u) < \eta/2$.  Let $u \in B(\omega, r)$ and observe that $\{U_i\}$ is a $(1/m)$-cover for $F_u$ giving that $\mathcal{H}^{G}_{1/m} (F_u) \leq \sum_i G(\lvert U_i \rvert) <\tfrac{1}{n}$.  It follows that $B(\omega, r) \subseteq H_{m,n}$ and that $H_{m,n}$ is open.
\\ \\
(ii) $H_{m,n}$ is dense.
\\ \\
Let  $\omega=(\omega_1, \omega_2, \dots) \in \Omega$ and $\varepsilon>0$.  Choose $k \in \mathbb{N}$ such that $2^{-k }<\varepsilon$ and choose $u = (u_1,u_2, \dots) \in \Omega$ such that 
\[
\mathcal{H}^G( F_u)  < \frac{1/n}{\lvert \mathcal{I}_{\omega_1} \rvert \cdots \lvert \mathcal{I}_{\omega_k} \rvert }.
\]
Let $v =(\omega_1,\dots, \omega_k, u_1,u_2, \dots)$.  It follows that $d_\Omega(\omega, v) < \varepsilon$ and, since
\[
F_v =\bigcup_{j_1\in \mathcal{I}_{\omega_1}, \dots, j_k \in\mathcal{I}_{\omega_k}} S_{\omega_1, j_1} \circ \dots \circ S_{\omega_k,j_k}(F_u),
\]
it follows that
\begin{eqnarray*}
\mathcal{H}^G_{1/m}(F_v) \  \leq \ \mathcal{H}^G(F_v) &=&  \mathcal{H}^G \Bigg(\bigcup_{j_1\in \mathcal{I}_{\omega_1}, \dots, j_k \in\mathcal{I}_{\omega_k}} S_{\omega_1, j_1} \circ \dots \circ S_{\omega_k,j_k} \big(F_{u} \big)\Bigg) \\ \\
&\leq& \sum_{j_1\in \mathcal{I}_{\omega_1}, \dots, j_k \in\mathcal{I}_{\omega_k}} \mathcal{H}^G\big(F_{u} \big)\\ \\
&\leq& \lvert \mathcal{I}_{\omega_1} \rvert \cdots \lvert \mathcal{I}_{\omega_k} \rvert  \, \mathcal{H}^G\big(F_{u} \big)\\ \\
&<&  1/n
\end{eqnarray*}
and so $v \in H_{m,n}$, proving that $H_{m,n}$ is dense. \hfill \qed

\subsection{Proof of Theorem \ref{main} (2)} \label{pmeasure}

Assume that $G$ is a doubling gauge and that $\sup_{u \in \Omega} \, \mathcal{P}^G(F_u) = \infty$.  We will show that the set
\[
P = \{ \omega \in \Omega : \mathcal{P}^G( F_\omega)  = \infty \}
\]
is residual.  The extra step in the definition of packing measure causes it to be more awkward to work with than Hausdorff measure.  To circumvent these difficulties we need the following two technical lemmas.

\begin{lma} \label{packingkey}
Suppose $F \subset K$ is such that for all open $V$ which intersect $F$, $\mathcal{P}_0^G(F \cap V) = \infty$.  Then $\mathcal{P}^G(F) = \infty$.
\end{lma}

\begin{proof}
Let $\{F_i\}_i$ be a countable sequence of closed sets such that $F \subset \cup_i F_i$.  The Baire Category Theorem implies that for some $i$ and some open set $V$, $F \cap V \subseteq F_i$ and hence $\mathcal{P}_0^G(F_i) = \infty$.  This means that, for every countable cover of $F$ by closed sets, at least one of the closed sets must have infinite packing pre-measure, proving the result.
\end{proof}

We will use Lemma \ref{packingkey} to prove the following Lemma, which will allow us to work with packing pre-measure instead of packing measure.

\begin{lma} \label{usepre}
We have $P = \{ \omega \in \Omega : \mathcal{P}^G_0( F_\omega)  = \infty \}$.
\end{lma}

\begin{proof}
It is clear that $P \subseteq \{ \omega \in \Omega : \mathcal{P}^G_0( F_\omega)  = \infty \}$.  We will now prove the opposite inclusion.  Let $\omega \in \Omega$ be such that $\mathcal{P}_0^G(F_\omega) = \infty$ and let $V$ be an open set which intersects $F_\omega$.  Choose $k$ large enough to ensure that for some $i_1 \in \mathcal{I}_{\omega_1}, \dots, i_k \in \mathcal{I}_{\omega_1}$ we have
\[
S_{\omega_1, i_1} \circ \cdots \circ S_{\omega_k, i_k} \big(F_{(\omega_{k+1}, \omega_{k+2}, \dots)}\big) \subseteq F \cap V.
\]
Write $\phi = S_{\omega_1, i_1} \circ \cdots \circ S_{\omega_k, i_k}$ and $u = (\omega_{k+1}, \omega_{k+2}, \dots)$.  Since packing pre-measure is finitely additive, we have
\begin{eqnarray*} \label{decomp}
\infty \ = \ \mathcal{P}_0^G(F_\omega)  &=& \mathcal{P}_0^G \Bigg( \bigcup_{i_1\in \mathcal{I}_{\omega_1}, \dots, i_k \in\mathcal{I}_{\omega_k}} S_{\omega_1, i_1} \circ \dots \circ S_{\omega_k,i_k} (F_u) \ \Bigg) \\ \\
&\leq&   \sum_{i_1\in \mathcal{I}_{\omega_1}, \dots, i_k \in\mathcal{I}_{\omega_k}} \mathcal{P}_0^G(F_u)  \\ \\
&\leq&  \lvert \mathcal{I}_{\omega_1} \rvert \cdots \lvert \mathcal{I}_{\omega_k} \rvert \ \mathcal{P}_0^G(F_u)
\end{eqnarray*}
and therefore
\begin{eqnarray*}
\mathcal{P}_0^G(F \cap V)  &\geq& \mathcal{P}_0^G(\phi(F_u)) \\ \\
&\geq& D^-\big(G, \text{Lip}^-(\phi)\big)\, \mathcal{P}_0^G(F_u) \\ \\
&=& \infty.
\end{eqnarray*}
Finally, by Lemma \ref{packingkey}, we have that $\mathcal{P}^G(F_\omega) = \infty$ and hence $\omega \in P$.
\end{proof}

Writing $P_{m,n}= \{ \omega \in \Omega : \mathcal{P}^{G}_{0, \, 1/m} (F_\omega) > n\}$, it follows from Lemma \ref{usepre} that
\[
P = \{ \omega \in \Omega : \mathcal{P}^G_0( F_\omega)  = \infty \} = \bigcap_{m, n \in \mathbb{N}} P_{m,n},
\]
so it suffices to prove that each $P_{m,n}$ is open and dense in $(\Omega, d_\Omega)$.  Fix $m, n \in \mathbb{N}$.
\\ \\
(i) $P_{m,n}$ is open.
\\ \\
Let $\omega \in P_{m,n}$.  It follows that there exists a finite centered $(1/m)$-packing of $F_\omega$ by closed balls, $\{U_i\}$, satisfying
\[
\sum_i G(\lvert U_i \rvert) >n.
\]
Let
\[
\eta = \min_{i \neq j} \min_{x \in U_i, y \in U_j} d(x,y)
\]
which is strictly positive since the sets $U_i$ are closed.  Now choose $r>0$ sufficiently small to ensure that, if $u \in B(\omega, r)$, then $d_\mathcal{H} (F_\omega, F_u) < \eta/2$ and fix such a $u \in B(\omega, r)$.  It follows that we can find a centered $(1/m)$-packing, $\{\tilde U_i\}$, of $F_u$, where $\tilde U_i$ is centered in $F_u$ and has the same diameter as $U_i$.  It follows that $\mathcal{P}^{G}_{0,  \,1/m} (F_u) \geq \sum_i G(\lvert U_i \rvert) >n$ and therefore $B(\omega, r) \subseteq P_{m,n}$, proving that $P_{m,n}$ is open.
\\ \\
(ii) $P_{m,n}$ is dense.
\\ \\
Let  $\omega=(\omega_1, \omega_2, \dots) \in \Omega$ and $\varepsilon>0$.  Choose $k \in \mathbb{N}$ such that $2^{-k }<\varepsilon$ and choose $u = (u_1,u_2, \dots) \in \Omega$ such that
\[
\mathcal{P}_0^G( F_u) \geq \frac{n}{\max_{j_1\in \mathcal{I}_{\omega_1}, \dots, j_k \in\mathcal{I}_{\omega_k}} D\big(G, \text{Lip}^-\big(S_{\omega_1, j_1} \circ \dots \circ S_{\omega_k,j_k}\big) \big)}
\]
Let $v =(\omega_1,\dots, \omega_k, u_1,u_2, \dots)$.  It follows that $d_\Omega(\omega, v) < \varepsilon$ and, since
\[
F_v =\bigcup_{j_1\in \mathcal{I}_{\omega_1}, \dots, j_k \in\mathcal{I}_{\omega_k}} S_{\omega_1, j_1} \circ \dots \circ S_{\omega_k,j_k}(F_u),
\]
it follows that
\begin{eqnarray*}
\mathcal{P}^G_{0, \, 1/m}(F_v) \ \geq \  \mathcal{P}_0^G(F_v) &=&  \mathcal{P}^G_0 \Bigg(\bigcup_{j_1\in \mathcal{I}_{\omega_1}, \dots, j_k \in\mathcal{I}_{\omega_k}} S_{\omega_1, j_1} \circ \dots \circ S_{\omega_k,j_k} (F_{u} )\Bigg) \\ \\
&\geq& \max_{j_1\in \mathcal{I}_{\omega_1}, \dots, j_k \in\mathcal{I}_{\omega_k}} \mathcal{P}^G_0 \big(S_{\omega_1, j_1} \circ \dots \circ S_{\omega_k,j_k} (F_{u} )\big)\\ \\
&\geq&  \max_{j_1\in \mathcal{I}_{\omega_1}, \dots, j_k \in\mathcal{I}_{\omega_k}} D\big(G, \text{Lip}^-\big(S_{\omega_1, j_1} \circ \dots \circ S_{\omega_k,j_k}\big) \big) \, \mathcal{P}^G_0 (F_{u} )\\ \\
&\geq&  n
\end{eqnarray*}
and so $v \in P_{m,n}$, proving that $P_{m,n}$ is dense. \hfill \qed

\subsection{Proof of Theorem \ref{main} (5)} \label{lbox}

It is well-known that lower box dimension is not finitely stable, see \cite{falconer}, Chapter 3, i.e., it is not true in general that $\underline{\dim}_\text{B} E \cup F \leq \max \{ \underline{\dim}_\text{B} E, \,  \underline{\dim}_\text{B} F\}$.  To get around this problem in the following proof, we begin with a simple technical lemma.

\begin{lma} \label{lowerstable}
Let $F \subset K$ be such that that $\underline{\dim}_\text{\emph{B}} F = s$ and let $\{\phi_i\}_{i \in \mathcal{S}}$ be a finite collection of Lipschitz contractions.  Then
\[
\underline{\dim}_\text{\emph{B}} \bigcup_{i \in \mathcal{S}} \phi_i(F) \leq s.
\]
\end{lma}

\begin{proof}
For all $\delta>0$ we have
\[
N_\delta \Big( \bigcup_{i \in \mathcal{S}} \phi_i(F)\Big) \leq  \sum_{i \in \mathcal{S}} N_\delta \big(\phi_i(F)\big) \leq \sum_{i \in \mathcal{S}} N_{\delta/\text{Lip}^+(\phi_i)}(F)
\leq \lvert \mathcal{S}\rvert  N_{\delta}(F).
\]
Taking logs, dividing by $-\log \delta$ and computing the limes inferior completes the proof.
\end{proof}

We now turn to the proof of Theorem \ref{main} (5).  Let $b=\inf_{u \in \Omega} \, \underline{\dim}_\text{{B}} F_u $.  We will show that the set
\[
B=\{ \omega \in \Omega : \underline{\dim}_\text{{B}} F_\omega \leq b\}
\]
is residual, from which Theorem \ref{main} (5) follows.  Writing
\[
B_n =\bigcup_{\delta \in (0, 1/n)} \Big\{ \omega \in \Omega : N_\delta(F_\omega) \leq \delta^{-(b+\tfrac{1}{n})} \Big\},
\]
we have
\[
B= \bigcap_{n \in \mathbb{N}} \ \bigcup_{\delta \in (0, 1/n)} \Big\{ \omega \in \Omega : \frac{\log N_\delta(F_\omega)}{-\log \delta} \leq b+\tfrac{1}{n} \Big\} = \bigcap_{n \in \mathbb{N}} B_n,
\]
so it suffices to prove that each $B_{n}$ is open and dense in $(\Omega, d_\Omega)$.  Fix $n \in \mathbb{N}$.
\\ \\
(i) $B_{n}$ is open.
\\ \\
Let $\omega \in B_n$.  It follows that for some $\delta<1/n$ there exists a $\delta$-cover of $F_\omega$ by fewer than $\delta^{-(b+\tfrac{1}{n})}$ open sets, $\{U_i\}$.  Let $\mathcal{U} = \partial  \big(\cup_i U_i\big)$ be the boundary of the union of the covering sets, $\{U_i\}$, and let
\[
\eta = \min_{x \in \mathcal{U}, y \in F_\omega} d(x,y)
\]
which is strictly positive by the compactness of $F_\omega$.  Now choose $r>0$ sufficiently small to ensure that if $u \in B(\omega, r)$, then $d_\mathcal{H} (F_\omega, F_u) < \eta/2$.  Let $u \in B(\omega, r)$ and observe that $\{U_i\}$ is a $(1/m)$-cover for $F_u$ giving that $N_\delta(F_\omega) \leq \delta^{-(b+\tfrac{1}{n})}$.  It follows that $B(\omega, r) \subseteq B_n$ and therefore $B_{n}$ is open.
\\ \\
(ii) $B_{n}$ is dense.
\\ \\
Let $\omega=(\omega_1, \omega_2, \dots) \in \Omega$ and $\varepsilon>0$.  Let  $u=(u_1, u_2, \dots) \in \Omega$ be such that $\underline{\dim}_\text{B} F_{u} \leq b+1/n$.  Now choose $k \in \mathbb{N}$ such that $2^{-k} <\varepsilon$ and let $v =(\omega_1,\dots, \omega_k,u_1, u_2, \dots)$.  It follows that $d_\Omega(v, \omega) < \varepsilon$ and, furthermore,
\[
F_v =\bigcup_{j_1\in \mathcal{I}_{\omega_1}, \dots, j_k \in\mathcal{I}_{\omega_k}} S_{\omega_1, j_1} \circ \dots \circ S_{\omega_k,j_k}(F_u).
\]
and since, for all $j_1\in \mathcal{I}_{\omega_1}, \dots, j_k \in\mathcal{I}_{\omega_k}$ the map $S_{\omega_1, j_1} \circ \dots \circ S_{\omega_k,j_k}$ is a Lipschitz contraction, it follows from Lemma \ref{lowerstable} that $\underline{\dim}_\text{B} F_{v} \leq \underline{\dim}_\text{B} F_{u} \leq b+1/n$ and so $v \in B_n$, proving that $B_n$ is dense. \hfill \qed

\subsection{Proof of Theorem \ref{infinite} (2)} \label{infiniteproof}

Write $h = \inf_{u \in \Omega} \,\dim_\H F_u$ and assume that $\inf_{u \in \Omega} \,\mathcal{H}^h( F_u)= \mathcal{H}_0 >0$, $v=(v_1, v_2, \dots) \in \Omega$ satisfies condition (\ref{lip-}) and that the RIFS satisfies the $\mathcal{H}^h$-MSC.  We will show the set
\[
M = \{ \omega \in \Omega : \mathcal{H}^h(F_\omega) < \infty\}
\]
is meagre, from which the result follows.  Writing $M_n = \{ \omega \in \Omega : \mathcal{H}^h(F_\omega) < n\}$, we have
\[
M  = \bigcup_{n \in \mathbb{N}} M_n,
\]
so it suffices to show that each $M_n$ is nowhere dense.  Fix $n \in \mathbb{N}$, $\omega \in M_n$ and $r>0$.  Now choose $k \in \mathbb{N}$ such that $2^{-k} < r$. It follows that the open ball $B_l = B \big( (\omega_1, \dots, \omega_k, v_1, v_2, \dots ), \, 2^{-l} \big)$ is contained in $B(\omega, r)$ for all $l>k$. Let $u \in B_l$, and note that
\[
u =  (\omega_1, \dots, \omega_k, v_1, \dots, v_{l-k}, u_1, u_2, \dots )
\]
for some $(u_1, u_2, \dots ) \in \Omega$.  Noting that the RIFS satisfies the $\mathcal{H}^h$-MSC and that $\text{Lip}^-$ is supermultiplicative, we have

\begin{eqnarray*}
\mathcal{H}^h(F_u) &=& \mathcal{H}^h \Bigg(\bigcup_{i_1\in \mathcal{I}_{\omega_1}, \dots, i_k \in\mathcal{I}_{\omega_k}} \bigcup_{j_1\in \mathcal{I}_{v_1}, \dots, j_{l-k} \in\mathcal{I}_{v_{l-k}}} S_{\omega_1, i_1} \circ \dots \circ S_{\omega_k,i_k} \circ S_{v_1, j_1} \circ \dots \circ S_{v_{l-k},j_{l-k}} \big(F_{(u_1, u_2, \dots)} \big)\Bigg) \\ \\
&=&  \sum_{i_1\in \mathcal{I}_{\omega_1}, \dots, i_k \in\mathcal{I}_{\omega_k}} \sum_{j_1\in \mathcal{I}_{v_1}, \dots, j_{l-k} \in\mathcal{I}_{v_{l-k}}}\mathcal{H}^h \Bigg( S_{\omega_1, i_1} \circ \dots \circ S_{\omega_k,i_k} \circ S_{v_1, j_1} \circ \dots \circ S_{v_{l-k},j_{l-k}} \big(F_{(u_1, u_2, \dots)} \big)\Bigg) \\ \\
&\geq&  \sum_{i_1\in \mathcal{I}_{\omega_1}, \dots, i_k \in\mathcal{I}_{\omega_k}} \sum_{j_1\in \mathcal{I}_{v_1}, \dots, j_{l-k} \in\mathcal{I}_{v_{l-k}}} \text{Lip}^-(S_{\omega_1, i_1} \circ \dots \circ S_{\omega_k,i_k} \circ S_{v_1, j_1} \circ \dots \circ S_{v_{l-k},j_{l-k}})^h \, \,  \mathcal{H}^h \big(F_{(u_1, u_2, \dots)} \big) \\ \\
&\geq&  \mathcal{H}_0 \,  \Bigg(\sum_{i_1\in \mathcal{I}_{\omega_1}, \dots, i_k \in\mathcal{I}_{\omega_k}}\text{Lip}^-(S_{\omega_1, i_1} \circ \dots \circ S_{\omega_k,i_k})^h\Bigg) \Bigg(  \sum_{j_1\in \mathcal{I}_{v_1}, \dots, j_{l-k} \in\mathcal{I}_{v_{l-k}}} \text{Lip}^-( S_{v_1, j_1} \circ \dots \circ S_{v_{l-k},j_{l-k}})^h \Bigg) \\ \\
&\to& \infty
\end{eqnarray*}
as $l \to \infty$.  It follows that we may choose $l$ large enough to ensure $B_l \subseteq B(\omega,r) \setminus M_n$ and so $M_n$ is nowhere dense.

\subsection{Proof of Theorem \ref{zero} (2)} \label{zeroproof}

Write $p = \sup_{u \in \Omega} \,\dim_\P F_u$ and assume that $\sup_{u \in \Omega} \,\mathcal{P}^p( F_u)= \mathcal{P}_0 < \infty$ and that $v=(v_1, v_2, \dots) \in \Omega$ satisfies condition (\ref{lip+}).  We will show the set
\[
N = \{ \omega \in \Omega : \mathcal{P}^h(F_\omega) >0\}
\]
is meagre, from which the result follows.  Writing $N_n =\{ \omega \in \Omega : \mathcal{P}^p(F_\omega) >1/n\}$, we have
\[
N  = \bigcup_{n \in \mathbb{N}} N_n,
\]
so it suffices to show that each $N_n$ is nowhere dense.  Fix $n \in \mathbb{N}$, $\omega \in N_n$ and $r>0$.  Now choose $k \in \mathbb{N}$ such that $2^{-k} < r$. It follows that the open ball $B_l = B \big( (\omega_1, \dots, \omega_k, v_1, v_2, \dots ), \, 2^{-l} \big)$ is contained in $B(\omega, r)$ for all $l>k$. Let $u \in B_l$, and note that
\[
u =  (\omega_1, \dots, \omega_k, v_1, \dots, v_{l-k}, u_1, u_2, \dots )
\]
for some $(u_1, u_2, \dots ) \in \Omega$.  Noting that $\text{Lip}^+$ is submultiplicative, we have

\begin{eqnarray*}
\mathcal{P}^p(F_u) &=& \mathcal{P}^p \Bigg(\bigcup_{i_1\in \mathcal{I}_{\omega_1}, \dots, i_k \in\mathcal{I}_{\omega_k}} \bigcup_{j_1\in \mathcal{I}_{v_1}, \dots, j_{l-k} \in\mathcal{I}_{v_{l-k}}} S_{\omega_1, i_1} \circ \dots \circ S_{\omega_k,i_k} \circ S_{v_1, j_1} \circ \dots \circ S_{v_{l-k},j_{l-k}} \big(F_{(u_1, u_2, \dots)} \big)\Bigg) \\ \\
&\leq&  \sum_{i_1\in \mathcal{I}_{\omega_1}, \dots, i_k \in\mathcal{I}_{\omega_k}} \sum_{j_1\in \mathcal{I}_{v_1}, \dots, j_{l-k} \in\mathcal{I}_{v_{l-k}}}\mathcal{P}^p \Bigg( S_{\omega_1, i_1} \circ \dots \circ S_{\omega_k,i_k} \circ S_{v_1, j_1} \circ \dots \circ S_{v_{l-k},j_{l-k}} \big(F_{(u_1, u_2, \dots)} \big)\Bigg) \\ \\
&\leq&  \sum_{i_1\in \mathcal{I}_{\omega_1}, \dots, i_k \in\mathcal{I}_{\omega_k}} \sum_{j_1\in \mathcal{I}_{v_1}, \dots, j_{l-k} \in\mathcal{I}_{v_{l-k}}} \text{Lip}^+(S_{\omega_1, i_1} \circ \dots \circ S_{\omega_k,i_k} \circ S_{v_1, j_1} \circ \dots \circ S_{v_{l-k},j_{l-k}})^p \, \,  \mathcal{P}^p \big(F_{(u_1, u_2, \dots)} \big) \\ \\
&\leq&  \mathcal{P}_0 \, \Bigg(  \sum_{i_1\in \mathcal{I}_{\omega_1}, \dots, i_k \in\mathcal{I}_{\omega_k}}\text{Lip}^+(S_{\omega_1,i_1} \circ \dots \circ S_{\omega_k,i_k})^p \Bigg) \Bigg(  \sum_{j_1\in \mathcal{I}_{v_1}, \dots, j_{l-k}\in\mathcal{I}_{v_{l-k}}} \text{Lip}^+( S_{v_1, j_1} \circ \dots \circ S_{v_{l-k},j_{l-k}})^p  \Bigg)   \\ \\
&\to& 0
\end{eqnarray*}
as $l \to \infty$.  It follows that we may choose $l$ large enough to ensure $B_l \subseteq B(\omega,r) \setminus N_n$ and so $N_n$ is nowhere dense.

\subsection{Proof of Theorem \ref{SS}} \label{SSproof}

The proof of Theorem \ref{SS} is a standard application of the mass distribution principle, Proposition \ref{MDP}.  Similar arguments can be found in, for example, \cite{falconer} Chapter 9.
\\ \\
For each $i \in D$, let $s_i$ be as in Section \ref{ssresults} and write $c = \min_{i \in D, \, j \in \mathcal{I}_i} \text{Lip}(S_{i,j})$. We will now define a mass distribution on $F_\omega$ which will be used in the subsequent proofs.  First define a measure, $\mu_\omega^{\text{sym}}$, on the symbollic space, $\prod_{l=1}^{\infty} \mathcal{I}_{\omega_l}$, by
\[
\mu_\omega^{\text{sym}}\Big( \big\{ (j_1, j_2, \dots) : j_1=i_1, \dots, j_k=i_k \big\} \Big) = \text{Lip}(S_{\omega_1, i_1})^{s_{\omega_1}} \cdots \text{Lip}( S_{\omega_k, i_k})^{s_{\omega_k}}
\]
for each $ (i_1, \dots, i_k) \in \prod_{l=1}^{k} \mathcal{I}_{\omega_l}$.  Now transfer $\mu_\omega^{\text{sym}}$ to a Borel probability measure $\mu_\omega$, supported on $F_\omega$, by
\[
\mu_\omega(E) = \mu_\omega^{\text{sym}}\bigg( \Big\{(i_1, i_2, \dots) \in \prod_{l=1}^{\infty} \mathcal{I}_{\omega_l} \ : \  \bigcap_{k} S_{\omega_1, i_1} \circ \dots \circ S_{\omega_k,i_k}(K) \in E \Big\} \bigg)
\]
for Borel sets $E \subseteq K$.
\\ \\
\emph{Proof of (1)}
\\ \\
Since each deterministic IFS satisfies the OSC, it is clear that $\sup_{\omega \in \Omega} \mathcal{P}^{s_{\max}}(F_\omega) \geq \sup_{\omega \in \Omega} \mathcal{H}^{s_{\max}}(F_\omega)>0$.  We will now show that $\sup_{\omega \in \Omega} \mathcal{P}^{s_{\max}}(F_\omega)< \infty$.  Fix $\omega = (\omega_1, \omega_2, \dots) \in \Omega$, let $x \in F_\omega$ and $r>0$.  Now let $l \in \mathbb{N}$ and $i_1\in \mathcal{I}_{\omega_1}, \dots, i_l \in \mathcal{I}_{\omega_l}$ be such that
\[
x \in S_{\omega_1, i_1} \circ \cdots \circ S_{\omega_l, i_l} (F_\omega)
\]
and
\[
\text{Lip}(S_{\omega_1, i_1}) \cdots \text{Lip}( S_{\omega_l, i_l}) \lvert K \rvert < r \leq \text{Lip}(S_{\omega_1, i_1}) \cdots \text{Lip}( S_{\omega_{l-1}, i_{l-1}}) \lvert K \rvert.
\]
It follows that
\begin{eqnarray*}
\mu_\omega (B(x,r)) \, r^{-s_{\max}} &\geq& \mu_\omega\Big(S_{\omega_1, i_1} \circ \cdots \circ S_{\omega_l, i_l} (F_\omega) \Big)  \, r^{-s_{\max}}\\ \\
&\geq& \text{Lip}(S_{\omega_1, i_1})^{s_{\omega_1}} \cdots \text{Lip}( S_{\omega_l, i_l})^{s_{\omega_l}}  \, r^{-s_{\max}}\\ \\
&\geq& \bigg(\frac{\text{Lip}(S_{\omega_1, i_1}) \cdots \text{Lip}( S_{\omega_l, i_l}) }{ r}\bigg)^{s_{\max}}\\ \\
&\geq& \bigg(\frac{r \, c \, \lvert K \rvert^{-1} }{ r}\bigg)^{s_{\max}}\\ \\
&=& \big(c / \lvert K \rvert \big)^{s_{\max}}
\end{eqnarray*}
and by Proposition \ref{MDP} (2) it follows that $\mathcal{P}^{s_{\max}}(F_\omega) \leq \big(2\,  \lvert K \rvert/c \big)^{s_{\max}} < \infty$ and, in particular,
\[
0 <  \sup_{\omega \in \Omega} \,  \mathcal{P}^{s_{\max}}(F_\omega)  < \infty.
\]

\emph{Proof of (3)}
\\ \\
We will need the following lemma which appears as Lemma 9.2 in \cite{falconer}.

\begin{lma} \label{balls}
Let $\{V_i\}$ be a collection of disjoint open subsets of $\mathbb{R}^n$ such that each $V_i$ contains a ball of radius $a_1r$ and is contained in a ball of radius $a_2r$.  Then any ball, $B$, of radius $r$ intersects at most $(1+2a_2)^na_1^{-n}$ of the closures, $\overline{V}_i$.
\end{lma}

Let $\mathcal{O}$ be the open set used in the UOSC and let $a_1, \, a_2$ be such that $\mathcal{O}$ contains a ball of radius $a_1$ and is contained in a ball of radius $a_2$.  Let $\mathcal{I}_\omega^* = \bigcup_{k \in \mathbb{N}}\prod_{l=1}^{k} \mathcal{I}_{\omega_l}$ and, for $r>0$, let $\mathcal{I}_\omega^r$ be an $r$-stopping defined by
\[
\mathcal{I}_\omega^r = \big\{ (i_1, i_2, \dots, i_l) \in \mathcal{I}_\omega^* \ : \ \text{Lip}(S_{\omega_1, i_1}) \cdots \text{Lip}( S_{\omega_l, i_l}) \leq r < \text{Lip}(S_{\omega_1, i_1}) \cdots \text{Lip}( S_{\omega_{l-1}, i_{l-1}}) \big\}.
\]
Note that
\begin{itemize}
\item[(1)]  $\big\{S_{\omega_1,i_1} \circ \cdots \circ S_{\omega_l,i_l}(\mathcal{O}): (i_1, i_2, \dots, i_l) \in \mathcal{I}_\omega^r \big\}$ is a collection of disjoint open subsets of $\mathbb{R}^n$;
\item[(2)]  Each $S_{\omega_1,i_1} \circ \cdots \circ S_{\omega_l,i_l}(\mathcal{O})$ contains a ball of radius $c \,a_1r$ and is contained in a ball of radius $a_2r$;
\item[(3)]  For each $(i_1, i_2, \dots, i_l) \in \mathcal{I}_\omega^r$, we have
\[
S_{\omega_1,i_1} \circ \cdots \circ S_{\omega_l,i_l}(F_{(\omega_{l+1}, \omega_{l+2}, \dots)}) \subseteq S_{\omega_1,i_1} \circ \cdots \circ S_{\omega_l,i_l}(\overline{\mathcal{O}}).
\]
\end{itemize}

Since each deterministic IFS satisfies the OSC, it is clear that $\inf_{\omega \in \Omega} \mathcal{H}^{s_{\min}}(F_\omega)< \infty$.  We will now show that $\inf_{\omega \in \Omega} \mathcal{H}^{s_{\min}}(F_\omega)> 0$.  Fix $\omega = (\omega_1, \omega_2, \dots) \in \Omega$, let $x \in F_\omega$ and $r>0$.  It follows from (1)--(3) and Lemma \ref{balls} that
\begin{eqnarray*}
\mu_\omega (B(x,r)) \, r^{-s_{\min}} &=& r^{-s_{\min}} \ \mu_\omega \big( B(x,r) \cap F \big) \\ \\
 &=& r^{-s_{\min}} \ \mu_\omega^{\text{sym}}  \bigg( \Big\{ (i_1, i_2, \dots) \in \prod_{l=1}^{\infty} \mathcal{I}_{\omega_l} \ : \   \bigcap_{k} S_{\omega_1, i_1} \circ \dots \circ S_{\omega_k,i_k}(K) \in  B(x,r) \cap F \Big\} \bigg) \\ \\
&\leq& r^{-s_{\min}} \ \mu_\omega^{\text{sym}} \Bigg( \ \bigcup_{\substack{
                  (i_1, i_2, \dots, i_l) \in \mathcal{I}_\omega^r:\\
                 B(x,r) \cap S_{\omega_1,i_1} \circ \cdots \circ S_{\omega_l,i_l}(\overline{\mathcal{O}}) \neq \emptyset
                }} \big\{ (j_1, j_2, \dots) : j_1=i_1, \dots, j_l=i_l \big\} \ \Bigg) \\ \\
&\leq& r^{-s_{\min}} \ \sum_{\substack{
                  (i_1, i_2, \dots, i_l) \in \mathcal{I}_\omega^r:\\
                 B(x,r) \cap S_{\omega_1,i_1} \circ \cdots \circ S_{\omega_l,i_l}(\overline{\mathcal{O}}) \neq \emptyset
                }} \text{Lip}(S_{\omega_1, i_1})^{s_{\omega_1}} \cdots \text{Lip}( S_{\omega_l, i_l})^{s_{\omega_l}} \\ \\
&\leq& r^{-s_{\min}} \ \Big(\text{Lip}(S_{\omega_1, i_1}) \cdots \text{Lip}( S_{\omega_l, i_l})\Big)^{s_{\min}} \ (1+2a_2)^n(c \,a_1)^{-n} \\ \\
&\leq& (1+2a_2)^n(c \,a_1)^{-n} \\ \\
&<& \infty
\end{eqnarray*}

and by Proposition \ref{MDP} (1) it follows that $\mathcal{H}^{s_{\min}}(F_\omega) \geq(1+2a_2)^{-n}(c \,a_1)^{n}>0$ and, in particular,
\[
0 < \inf_{\omega \in \Omega} \,  \mathcal{H}^{s_{\min}}(F_\omega)  < \infty
\]
which completes the proof. \hfill \qed

\subsection{Proof of Theorem \ref{SS2} (1)} \label{SSproof2}

Write $\mathcal{H}_{\min} =\inf_{\omega \in \Omega} \,  \mathcal{H}^{s_{\min}}(F_\omega)$ and $\mathcal{P}_{\max} =  \sup_{\omega \in \Omega} \,  \mathcal{P}^{s_{\max}}(F_\omega) $ and let $s=s_{\min} = s_{\max}$.
\\ \\
\textbf{Hausdorff measure}
\\ \\
We will show that the set
\[
H = \{ \omega \in \Omega : \mathcal{H}^s( F_\omega) =\mathcal{H}_{\min}\}
\]
is residual.  Writing $H_{m,n} = \{ \omega \in \Omega : \mathcal{H}^{s}_{1/m} (F_\omega) < \mathcal{H}_{\min} + \tfrac{1}{n}\}$, we have
\begin{eqnarray*}
H= \bigcap_{m, n \in \mathbb{N}} H_{m,n},
\end{eqnarray*}
so it suffices to prove that each $H_{m,n}$ is open and dense in $(\Omega, d_\Omega)$.  Fix $m, n \in \mathbb{N}$.  It can be shown that $H_{m,n}$ is open using a similar approach to that used in the proof of Theorem \ref{main} (1).  We will now prove that $H_{m,n}$ is dense.
\\ \\
Let  $\omega=(\omega_1, \omega_2, \dots) \in \Omega$ and $\varepsilon>0$.  Choose $k \in \mathbb{N}$ such that $2^{-k }<\varepsilon$ and choose $u = (u_1,u_2, \dots) \in \Omega$ such that 
\[
\mathcal{H}^s( F_u)  < \mathcal{H}_{\min} + \tfrac{1}{n}.
\]
Let $v =(\omega_1,\dots, \omega_k, u_1,u_2, \dots)$.  It follows that $d_\Omega(\omega, v) < \varepsilon$ and, furthermore,
\begin{eqnarray*}
\mathcal{H}^s_{1/m}(F_v) \  \leq \ \mathcal{H}^s(F_v) &=&  \mathcal{H}^s \Bigg(\bigcup_{j_1\in \mathcal{I}_{\omega_1}, \dots, j_k \in\mathcal{I}_{\omega_k}} S_{\omega_1, j_1} \circ \dots \circ S_{\omega_k,j_k} \big(F_{u} \big)\Bigg) \\ \\
&\leq& \sum_{j_1\in \mathcal{I}_{\omega_1}, \dots, j_k \in\mathcal{I}_{\omega_k}} \text{Lip}(S_{\omega_1, j_1})^s \cdots  \text{Lip}(S_{\omega_k,j_k})^s \ \mathcal{H}^s\big(F_{u} \big)\\ \\
&<& \Big(\mathcal{H}_{\min} + \tfrac{1}{n}\Big) \  \sum_{j_1\in \mathcal{I}_{\omega_1}, \dots, j_k \in\mathcal{I}_{\omega_k}} \text{Lip}(S_{\omega_1, j_1})^s \cdots  \text{Lip}(S_{\omega_k,j_k})^s\\ \\
&=& \mathcal{H}_{\min} + \tfrac{1}{n}
\end{eqnarray*}
where the final equality is due to the fact that $s$ is a solution to Hutchison's formula for each deterministic IFS.  It follows that $v \in H_{m,n}$, proving that $H_{m,n}$ is dense.
\\ \\
\textbf{Packing measure}
\\ \\
We will show that the set $P = \{ \omega \in \Omega : \mathcal{P}^s( F_\omega)  = \mathcal{P}_{\max} \}$ is residual.  It was proved in \cite{packingmeasure} that if a compact set has finite packing pre-measure, then the packing measure and packing pre-measure coincide.  Writing $P_{m,n} = \{ \omega \in \Omega : \mathcal{P}^{s}_{0, \, 1/m} (F_\omega) > \mathcal{P}_{\max}-\tfrac{1}{n}\}$, it follows that
\[
P \supseteq \{ \omega \in \Omega : \mathcal{P}^s_0( F_\omega)  = \mathcal{P}_{\max} \}   = \bigcap_{m, n \in \mathbb{N}} P_{m,n},
\]
so it suffices to prove that each $P_{m,n}$ is open and dense.  Fix $m, n \in \mathbb{N}$.  It can be shown using a similar approach to that used in the proof of Theorem \ref{main} (2) that $P_{m,n}$ is open.  We will now show that it is also dense.
\\ \\
Let  $\omega=(\omega_1, \omega_2, \dots) \in \Omega$ and $\varepsilon>0$.  Choose $k \in \mathbb{N}$ such that $2^{-k }<\varepsilon$ and choose $u = (u_1,u_2, \dots) \in \Omega$ such that $\mathcal{P}^s( F_u)> \mathcal{P}_{\max}-\tfrac{1}{n}$.  Let $v =(\omega_1,\dots, \omega_k, u_1,u_2, \dots)$.  It follows that $d_\Omega(\omega, v) < \varepsilon$ and, furthermore,
\begin{eqnarray*}
\mathcal{P}^s_{0,1/m}(F_v) \  \geq \ \mathcal{P}^s(F_v) &=&  \mathcal{P}^s \Bigg(\bigcup_{j_1\in \mathcal{I}_{\omega_1}, \dots, j_k \in\mathcal{I}_{\omega_k}} S_{\omega_1, j_1} \circ \dots \circ S_{\omega_k,j_k} \big(F_{u} \big)\Bigg) \\ \\
&=& \sum_{j_1\in \mathcal{I}_{\omega_1}, \dots, j_k \in\mathcal{I}_{\omega_k}} \text{Lip}(S_{\omega_1, j_1})^s \cdots  \text{Lip}(S_{\omega_k,j_k})^s \ \mathcal{P}^s\big(F_{u} \big)\\ \\
&>& \Big(\mathcal{P}_{\max} - \tfrac{1}{n} \Big) \ \sum_{j_1\in \mathcal{I}_{\omega_1}, \dots, j_k \in\mathcal{I}_{\omega_k}} \text{Lip}(S_{\omega_1, j_1})^s \cdots  \text{Lip}(S_{\omega_k,j_k})^s \\ \\
&=& \mathcal{P}_{\max} - \tfrac{1}{n}
\end{eqnarray*}
where the final equality is due to the fact that $s$ is a solution to Hutchison's formula for each deterministic IFS.  It follows that $u \in P_{m,n}$, proving that $P_{m,n}$ is dense. \hfill \qed

\section{Examples} \label{examples}

In this section we provide a number of examples designed to illustrate some of the key points made in Section \ref{results}.  The examples in Sections \ref{conjfalse} and \ref{affineexample} will be random Sierpi\'nski carpets, as discussed in Section \ref{probapp}.

\subsection{Typical Hausdorff and packing measure} \label{conjfalse}

In this section we give two simple examples which show that the Hausdorff measure can typically be positive and finite even if the supremal Hausdorff measure is infinite and the packing measure can typically be positive and finite even if the infimal packing measure is zero.  The existence of these examples is slightly surprising in view of Theorems \ref{infinite} and \ref{zero} and the behaviour observed in the self-similar setting, see Theorem \ref{SS2}.
\\ \\
\textbf{Hausdorff measure}
\\ \\
Let $\mathbb{I} = \{\mathbb{I}_1, \mathbb{I}_2\}$ be a RIFS where $\mathbb{I}_1$ and $\mathbb{I}_2$ are IFSs of orientation preserving affine self-maps on $[0,1]^2$ corresponding to the figure below. 

\begin{figure}[H]
	\centering
	\includegraphics[width=80mm]{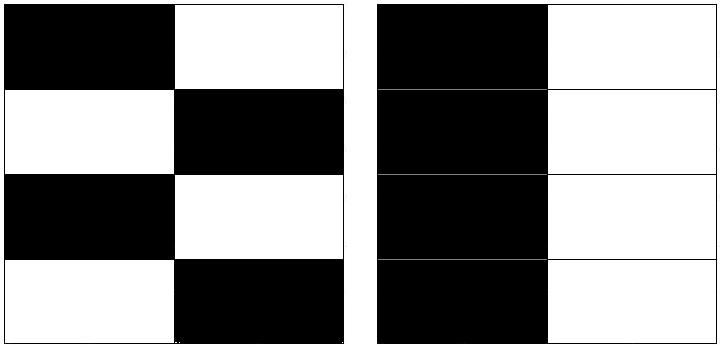}
\caption{The defining pattern for a random Sierpi\'nski carpet with $N=2,\  m_1=m_2=2$ and $n_1=n_2=4$.}
\end{figure}

It is clear that $\inf_{\omega \in \Omega} \dim_\H F_\omega = 1$ and $\inf_{\omega \in \Omega} \mathcal{H}^1( F_\omega) = 1< \infty =\sup_{\omega \in \Omega} \mathcal{H}^1( F_\omega) $.  It follows from Theorem \ref{main} that the typical Hausdorff dimension is 1.  We will now show that the typical Hausdorff measure is also infimal and, in particular, positive and finite.  We will show that the set $H = \{ \omega \in \Omega : \mathcal{H}^1( F_\omega) =1\}$ is a dense $G_\delta$ set and thus residual.  It can be shown that $H$ is $G_\delta$ using a very similar approach to that used in the proof of Theorem \ref{main} (1).  It remains to show that $H$ is dense.
\\ \\
Let  $\omega=(\omega_1, \omega_2, \dots) \in \Omega$ and $\varepsilon>0$.  Choose $k \in \mathbb{N}$ such that $2^{-k }<\varepsilon$ and let $v =(\omega_1,\dots, \omega_k,2,2, \dots)$.  It follows that $d_\Omega(\omega, v) < \varepsilon$ and, furthermore, since $F_{(2,2,\dots)} = \{0\}\times[0,1]$, we have
\[
F_v =\bigcup_{j_1\in \mathcal{I}_{\omega_1}, \dots, j_k \in\mathcal{I}_{\omega_k}} S_{\omega_1, j_1} \circ \dots \circ S_{\omega_k,j_k}\big(\{0\}\times[0,1] \big)
\]
and, since the vertical component of every map in $\mathbb{I}$ is a similarity with contraction ratio $1/4$ and both deterministic IFSs consist of 4 maps, we have
\[
\mathcal{H}^1(F_v) \leq \sum_{j_1\in \mathcal{I}_{\omega_1}, \dots, j_k \in\mathcal{I}_{\omega_k}} \mathcal{H}^1 \Bigg( S_{\omega_1, j_1} \circ \dots \circ S_{\omega_k,j_k} \big(\{0\}\times[0,1] \big)\Bigg)
=4^k \ 4^{-k} \ \mathcal{H}^1 \big(\{0\}\times[0,1] \big)
=  1
\]
and so $v \in H$, proving that $H$ is dense. 
\\ \\
\textbf{Packing measure}
\\ \\
Let $\mathbb{I} = \{\mathbb{I}_1, \mathbb{I}_2\}$ be a RIFS where $\mathbb{I}_1$ and $\mathbb{I}_2$ are IFSs of orientation preserving affine self-maps on $[0,1]^2$ corresponding to the figure below.

\begin{figure}[H]
	\centering
	\includegraphics[width=80mm]{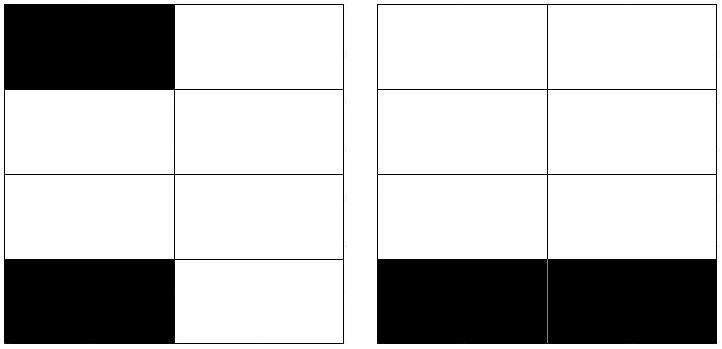}
\caption{The defining pattern for a random Sierpi\'nski carpet with $N=2, \ m_1=m_2=2$ and $n_1=n_2=4$.}
\end{figure}

We claim that $\inf_{\omega \in \Omega} \mathcal{P}^1( F_\omega) = 0 <1 \leq \sup_{\omega \in \Omega} \mathcal{P}^1( F_\omega) \leq 4$ and it follows that $\sup_{\omega \in \Omega} \dim_\P F_\omega = 1$.  The only inequality which is not obvious is $\sup_{\omega \in \Omega} \mathcal{P}^1( F_\omega) \leq 4$ which we will now prove.  Fix $\omega \in \Omega$ and define a mass distribution, $\mu_\omega$, on $F_\omega$ by assigning each level $k$ rectangle mass $2^{-k}$ in a similar way to the construction of the measures in Section \ref{SSproof}. It is easy to see that for all $x \in F_\omega$ we have $\liminf_{r \to 0} \mu(B(x,r)) \, r^{-1} \geq 1/2$ and it follows from Proposition \ref{MDP} (2) that $\mathcal{P}^1( F_\omega) \leq 4$.  Theorem \ref{main} gives that the typical packing dimension is 1.  We will now show that the typical packing measure is greater than or equal to 1 and, in particular, positive and finite.  We will show that the set $P = \{ \omega \in \Omega : \mathcal{P}^1( F_\omega)  \geq 1 \}$ is a dense $G_\delta$ set and thus residual.  It follows from the result in \cite{packingmeasure} and Lemma \ref{usepre} that $P = \{ \omega \in \Omega : \mathcal{P}^1_0( F_\omega)  \geq 1 \}$ and it can thus be shown that $P$ is $G_\delta$ using a very similar approach to that used in the proof of Theorem \ref{main} (2).  It remains to show that $P$ is dense.
\\ \\
Let  $\omega=(\omega_1, \omega_2, \dots) \in \Omega$ and $\varepsilon>0$.  Choose $k \in \mathbb{N}$ such that $2^{-k }<\varepsilon$ and let $v =(\omega_1,\dots, \omega_k, 2,2, \dots)$.  It follows that $d_\Omega(\omega, v) < \varepsilon$ and, furthermore, since $F_{(2,2,\dots)} = [0,1]\times\{0\}$, we have
\[
F_v =\bigcup_{j_1\in \mathcal{I}_{\omega_1}, \dots, j_k \in\mathcal{I}_{\omega_k}} S_{\omega_1, j_1} \circ \dots \circ S_{\omega_k,j_k}\big([0,1]\times\{0\}\big)
\]
and, since the horizontal component of every map in $\mathbb{I}$ is a similarity with contraction ratio $1/2$ and both deterministic IFSs consist of 2 maps, we have
\[
\mathcal{P}^1_{0}(F_v) \ = \  \mathcal{P}^1(F_v) = \sum_{j_1\in \mathcal{I}_{\omega_1}, \dots, j_k \in\mathcal{I}_{\omega_k}} \mathcal{P}^1  \bigg(S_{\omega_1, j_1} \circ \dots \circ S_{\omega_k,j_k} \big([0,1]\times\{0\}\big)\bigg) = 2^k \ 2^{-k} \ \mathcal{P}^1\big([0,1]\times\{0\} \big) =  1
\]
and so $u \in P$, proving that $P$ is dense.

\begin{rem} 
We believe that a more delicate application of the mass distribution principle will yield that, in fact, $\sup_{\omega \in \Omega} \mathcal{P}^1( F_\omega) =1$, but since the important thing for our purposes is that the typical value is positive and finite, we omit further calculation.
\end{rem}

\subsection{Dimension outside range} \label{affineexample}

In this section we give a simple example which shows that in the non-conformal setting the dimension of the random attractor need not be bounded below by the minimum dimension of the deterministic attractors.  This is in stark contrast to Theorem \ref{SS}, concerning random self-similar sets.  Furthermore, $\inf_{u \in \Omega} \dim_\H F_u$ is not attained by any finite combination of the determinsitic IFSs.  Let $\mathbb{I} = \{\mathbb{I}_1, \mathbb{I}_2\}$ be a RIFS where $\mathbb{I}_1$ and $\mathbb{I}_2$ are IFSs of orientation preserving affine self-maps on $[0,1]^2$ corresponding to the figure below.

\begin{figure}[H]
	\centering
	\includegraphics[width=80mm]{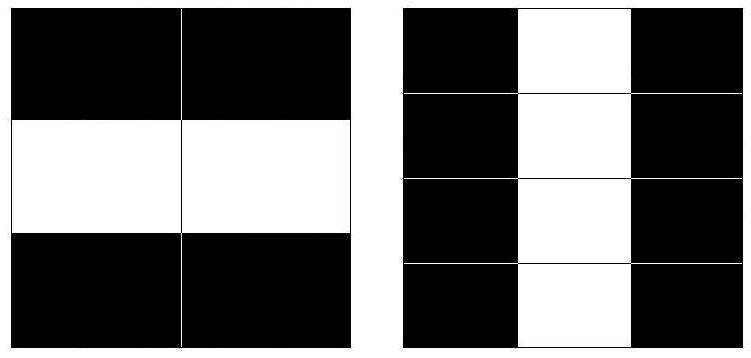}
\caption{The defining pattern for a random Sierpi\'nski carpet with $N=2,\,  m_1=2,\,  n_1=3,\,  m_2=3$ and $n_2=4$.}
\end{figure}

The results of \cite{bedford, mcmullen} give that for both deterministic attractors the Hausdorff, box and packing dimensions are all equal to $1+ \log 2 / \log 3 \approx 1.63$.  For $p \in [0,1]$, associate a probability vector $(p,\,1-p)$ with this system.  By the result of \cite{me_random}, given here as Theorem \ref{melars}, the almost sure Hausdorff dimension of $F_\omega$ is given by
\begin{eqnarray*}
\dim_\H F_\omega &=& \frac{p}{\log 2^p  3^{1-p}} \log\bigg( 2^{\log 2^p  3^{1-p} / \log 3^p  4^{1-p}}+2^{\log 2^p  3^{1-p} / \log 3^p  4^{1-p}} \bigg)\\ \\ 
&\,& \qquad+\frac{1-p}{\log 2^p  3^{1-p}} \log\bigg( 4^{\log 2^p  3^{1-p} / \log 3^p  4^{1-p}}+4^{\log 2^p  3^{1-p} / \log 3^p  4^{1-p}} \bigg)\\ \\
&=& \frac{\log 2}{\log 2^p  3^{1-p}} + (2-p) \, \frac{\log 2}{\log 3^p  4^{1-p}}.
\end{eqnarray*}
In fact, since each deterministic IFS has \emph{uniform vertical fibres} it follows from results in \cite{random_carpets2} that the above formula also gives the almost sure box and packing dimensions of $F_\omega$.  Plotting this as a function of $p$, we obtain
\begin{figure}[H]
	\centering
	\includegraphics[width=60mm]{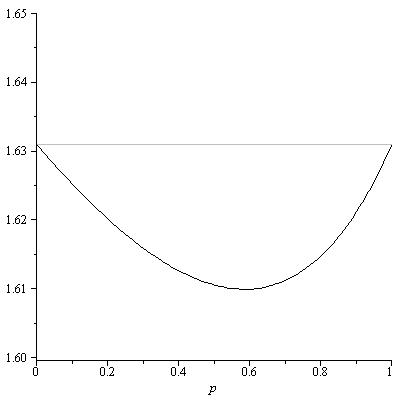}
		\caption{A graph of the almost sure Hausdorff dimension as a function of $p$.  The grey line shows the dimension of the deterministic attractors.}
\end{figure}

Notice the nonlinear dependence on $p$ and the fact that for $p \in (0,1)$ the almost sure dimension is lower than the minimum dimension of the two deterministic attractors.  In particular, the dimension of $F_\omega$ is not bounded below by the minimum Hausdorff dimension of the deterministic attractors, despite the fact that the UOSC is satisfied.  As such it is not at all clear what the infimal (and thus typical) Hausdorff dimension is.  This is in stark contrast to the self-similar setting, see Theorem \ref{SS} (4).  It is natural to ask if the infimal dimension is attained by an attractor of a deterministic IFS given by a finite combination of the original deterministic IFSs, $\mathbb{I}_1$ and $\mathbb{I}_2$.  We will argue now that it is not.  Finite combinations of $\mathbb{I}_1, \mathbb{I}_2$ give deterministic IFSs with attractors equal to $F_\omega$ for some `rational' $\omega \in \Omega$, i.e., some $\omega$ which consists of a finite word over $D$ repeated infinitely often.  Fix such a finite combination and let $N_1$ be the number of times we have used $\mathbb{I}_1$ and let $N_2$ be the number of times we have used $\mathbb{I}_2$.  It is clear, and in fact it follows from the results in \cite{random_carpets2}, that the Hausdorff dimension of the attractor is equal to the almost sure Hausdorff dimension of the attractor corresponding to $p = N_1/(N_1+N_2) \in \mathbb{Q}$.  However, elementary optimisation reveals that the minimum almost sure Hausdorff dimension (seen as the minimum of the graph above)
is attained by $p=2-\sqrt{2} \notin \mathbb{Q}$.

\subsection{Typical measure not positive and finite} \label{notdense}

In this section we will give a straightforward example which has the interesting property that, although the Hausdorff and packing measures of the attractors of the deterministic IFSs in the appropriate dimension are positive and finite, the typical Hausdorff and packing measures are infinity and zero, respectively.
\\ \\
Let $S_1, S_2, S_3: [0,1] \to [0,1]$ be defined by
\[
S_1(x) = x/3, \qquad S_2(x) = x/3+1/3, \qquad \text{and} \qquad S_3(x) = x/3+2/3.
\]
Let $\mathbb{I}$ be the RIFS consisting of the two deterministic IFSs, $\{S_1,S_3\}$ and $\{S_1, S_2,S_3\}$.  The attractors for these systems are the middle $1/3$ Cantor set, $C_{1/3}$, and the unit interval, $[0,1]$, respectively.  Also, since the first IFS is contained in the second, for all $\omega \in \Omega$,
\[
C_{1/3} \subseteq F_\omega \subseteq [0,1]
\]
from which it follows that dimensions are bounded between $s=\tfrac{\log 2}{\log 3}$ and 1 and that
\[
\inf_{u \in \Omega} \,\mathcal{H}^s( F_u) = \mathcal{H}^s(C_{1/3}) = 1
\]
and
\[
\sup_{u \in \Omega} \,\mathcal{P}^1( F_u) = \mathcal{P}^1([0,1]) = 1.
\]
It follows from Theorem \ref{SS2}  that, for a typical $\omega \in \Omega$, the set $F_\omega$ has Hausdorff and lower box dimension equal to $\tfrac{\log 2}{\log 3}$ and packing and upper box dimension equal to 1 but $\tfrac{\log 2}{\log 3}$-dimensional Hausdorff measure equal to $\infty$ and 1-dimensional packing measure equal to $0$.  It is clear that the $\mathcal{P}^{\log 2/ \log3}$-MSC is satisfied.

\subsection{A nonlinear example: random cookie cutters}

Although the previous examples illustrate some of the key phenomenon we wish to discuss, they have all been based on RIFSs consisting of translate linear (affine) maps.  Of course, Theorems \ref{main}, \ref{infinite} and \ref{zero} apply in far more general circumstances than this.  In this section we construct a more complicated example using nonlinear maps to which we can apply Theorems \ref{infinite} and \ref{zero} to deduce that neither the typical Hausdorff nor packing measures are positive and finite in the appropriate dimensions.
\\ \\
Let $f_1, f_2:[0,1] \to \mathbb{R}$ be defined by
\[
f_1(x) = -5x(x-1) \qquad \text{ and } \qquad f_2(x) = 9(x-1/6)(x-5/6)
\]
respectively.
\begin{figure}[H]
	\centering
	\includegraphics[width=50mm]{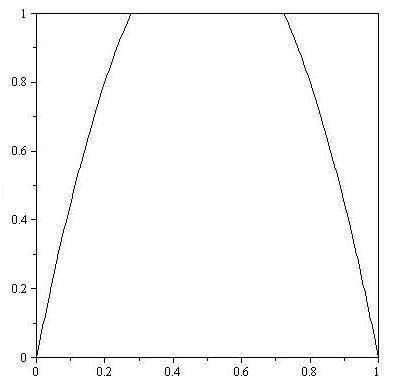}
	\qquad
	\includegraphics[width=50mm]{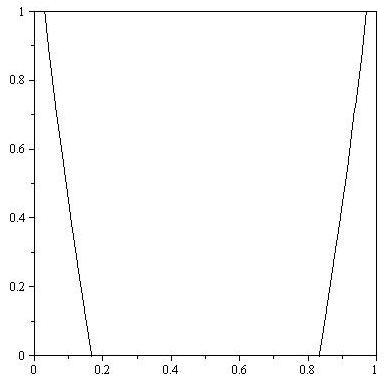}
	\caption{Graphs of the maps $f_1$ (left) and $f_2$ (right) restricted to the unit square.}
\end{figure}
Observe that $f$ maps each of the intervals $X_{1,1} = \big[0,\, \tfrac{1}{2}-\tfrac{1}{10}\sqrt{5}\big]$ and $X_{1,2} = \big[\tfrac{1}{2}+\tfrac{1}{10}\sqrt{5}, \, 1 \big]$ bijectively onto $[0,1]$ and furthermore $f_1'$ is continuous and
\begin{equation} \label{exp1}
2\leq \lvert f_1'(x) \rvert \leq  5
\end{equation}
for $x \in X_{1,1} \cup X_{1,2}$.  Similarly, $f_2$ maps each of the intervals $X_{2,1} = \big[\tfrac{1}{2}-\tfrac{1}{3}\sqrt{2}, \, \tfrac{1}{6} \big]$ and $X_{2,2} = \big[\tfrac{5}{6}, \, \tfrac{1}{2}+\tfrac{1}{3}\sqrt{2},\big]$ bijectively onto $[0,1]$, $f_2'$ is continuous and
\begin{equation} \label{exp2}
6 \leq \lvert f_2'(x) \rvert \leq 9
\end{equation}
for $x \in X_{2,1}\cup X_{2,2}$.  We have constructed two expanding dynamical systems $(X_{1,1} \cup X_{1,2}, f_1)$ and $(X_{2,1}\cup X_{2,2}, f_2)$ with repellers given by
\[
F_1 = \bigcap_{k \geq 0} f_1^{-k}\big([0,1]\big)  \qquad \text{ and } \qquad 
F_2 = \bigcap_{k \geq 0} f_2^{-k}\big([0,1]\big)
\]
respectively.  Repellers of this type are often called \emph{cookie cutters} and the Hausdorff dimension can be computed via the thermodynamical formalism.  For a more detailed account of cookie cutters and the thermodynamical formalism, the reader is referred to \cite{techniques}, Chapters 4--5.  We can view $F_1$ and $F_2$ as attractors of deterministic IFSs consisting of the inverse branches of $f_1$ and $f_2$.  In particular, the inverse branches of $f_1$ are given by 
\[
S_{1,1}(x) = \tfrac{1}{2}-\tfrac{1}{2}\sqrt{1-\tfrac{4}{5}x}  \,\Big)  \qquad \text{ and } \qquad S_{1,2}(x) = \tfrac{1}{2}+\tfrac{1}{2}\sqrt{1-\tfrac{4}{5}x}  \,\Big)
\]
and the inverse branches of $f_2$ are given by
\[
S_{2,1}(x) = \tfrac{1}{2}-\tfrac{1}{3}\sqrt{1+x}    \qquad \text{ and } \qquad S_{2,2}(x) = \tfrac{1}{2}+\tfrac{1}{3}\sqrt{1+x}.
\]
Let $\mathbb{I}$ be the RIFS consisting of $\mathbb{I}_1 = \{S_{1,1}, S_{1,2}\}$ and $\mathbb{I}_1 = \{S_{2,1}, S_{2,2}\}$.  Here $F_1$ corresponds to the choice $(1,1, \dots) \in \Omega$ and $F_2$ corresponds to the choice $(2,2, \dots) \in \Omega$.  For an arbitrary $\omega = (\omega_1,\omega_2, \dots) \in \Omega$, we obtain a \emph{random cookie cutter}
\[
F_\omega = \bigcap_{k \geq 0} f_{\omega_1}^{-1} \circ \cdots \circ f_{\omega_k}^{-1} \big([0,1]\big).
\]
Write $h= \inf_{u \in \Omega} \, \dim_\H F_u$ and $p= \sup_{u \in \Omega} \, \dim_\P F_u$.  It follows from (\ref{exp1}-\ref{exp2}), the fact that $f_1', f_2'$ are continuous and the mean value theorem that, for $i=1,2$,
\[
1/5 \leq \text{Lip}^-(S_{1,i})  \leq \text{Lip}^+(S_{1,i})   \leq 1/2 \qquad \text{ and } \qquad 1/9  \leq \text{Lip}^-(S_{2,i}) \leq \text{Lip}^+(S_{2,i})   \leq 1/6
\]
and applying standard estimates for the dimension gives
\[
h \leq \dim_\H F_2 \leq \frac{\log 2}{\log 6} < \frac{\log 2}{\log 5} \leq \dim_\P F_1 \leq p,
\]
see \cite{falconer}, Propositions 9.6--9.7.  Furthermore,
\[
\sum_{i_1\in \mathcal{I}_{1}, \dots, i_k \in\mathcal{I}_{1}} \text{Lip}^{-}(S_{1, i_1} \circ \dots \circ S_{1,i_k})^h \geq \big( 2 \cdot  5^{-h}\big) ^k  \to \infty
\]
and
\[
\sum_{j_1\in \mathcal{I}_{2}, \dots, j_k \in\mathcal{I}_{2}} \text{Lip}^{+}(S_{2, j_1} \circ \dots \circ S_{2,j_k})^p \leq \big( 2 \cdot 6^{-p} \big) ^k \to 0
\]
as $k \to \infty$.  It follows from Theorem \ref{main}, \ref{infinite} and \ref{zero} that, for a typical $\omega \in \Omega$, $\dim_\H F_\omega = h < p = \dim_\P F_\omega$ but 
\[
\mathcal{H}^h(F_\omega) = \left\{ \begin{array}{cc}
0 &  \text{if $\inf_{u \in \Omega} \,\mathcal{H}^h( F_u) = 0$}\\ \\
\infty  &  \text{if $\inf_{u \in \Omega} \,\mathcal{H}^h( F_u) >0$}
\end{array} \right. 
\]
and
\[
\mathcal{P}^p(F_\omega) = \left\{ \begin{array}{cc}
0 &  \text{if $\sup_{u \in \Omega} \,\mathcal{P}^p( F_u) < \infty$}\\ \\
\infty  &  \text{if $\sup_{u \in \Omega} \,\mathcal{P}^p( F_u) = \infty$}
\end{array} \right. 
\]
In particular, for a typical $\omega \in \Omega$, the random cookie cutter $F_\omega$ is `dimensionless' in the sense that neither the $s$-dimensional Hausdorff measure nor the $s$-dimensional packing measure are positive and finite for any $s \geq 0$.

\subsection{Some pictorial examples} \label{pics}

In this section we give some pictorial examples of attractors of RIFSs to illustrate some of the rich and complicated structures we can expect to see.  Although our results apply in both examples we do 
not perform any calculations.
\\ \\
Let $S_1, S_2, S_3:\mathbb{R}^2\to \mathbb{R}^2$ be defined by $S_1(x,y) = (x/2, y^2/2)$, $S_2(x,y) = (x/2+1/2, y/2)$ and $S_3(x,y)=(x^2/2, y/2+1/2)$ and let $T_1, T_2:\mathbb{R}^2\to \mathbb{R}^2$ be defined by $T_1(x,y) = (x^2/2, y/2)$ and $T_2(x,y) = (x/2+1/2, y/2+1/2)$.  Finally, let $\mathbb{I}$ be the RIFS consisting of $\mathbb{I}_1 =\{S_1, S_2, S_3\}$ and $\mathbb{I}_2 = \{T_1, T_2\}$.

\begin{figure}[H]
	\centering
	\includegraphics[width=50mm]{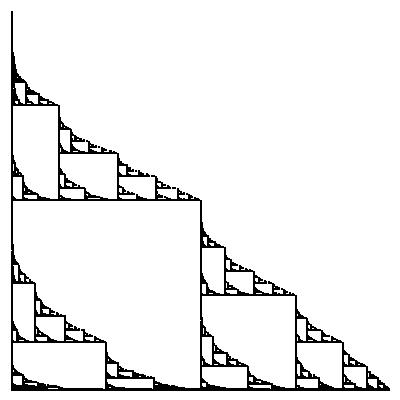}
	\qquad
	\includegraphics[width=50mm]{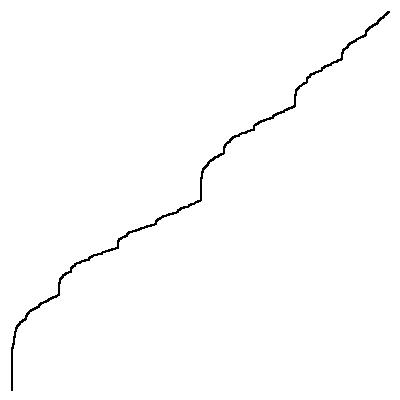} \\ \vspace{5mm}
	\includegraphics[width=50mm]{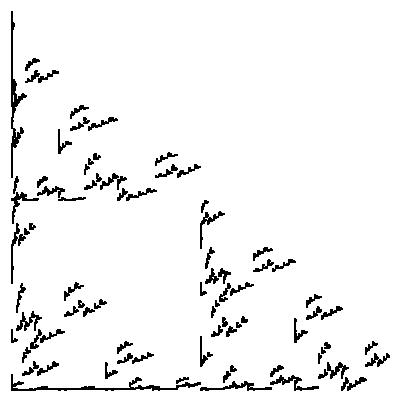}
	\qquad
	\includegraphics[width=50mm]{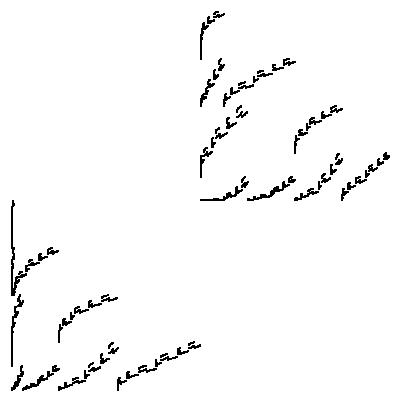}
	\caption{The attractors of $\mathbb{I}_1$ (top left) and $\mathbb{I}_2$ (top right) along with two random attractors of $\mathbb{I}$ corresponding to $\omega = (1,1,1,2,1,2, \dots)$ (bottom left) and  $\omega = (2,1,1,2,2,2,\dots)$ (bottom right).}
\end{figure}

Let $U_1, U_2, U_3:\mathbb{R}^2\to \mathbb{R}^2$ be defined by $U_1(x,y) = (x/3+y/6, y/3)$, $U_2(x,y) = (x/3-y/6+1/6, y/3+2/3)$ and $U_3(x,y)=(x/2+1/2, y/3+1/3)$ and let $V_1, V_2,V_3:\mathbb{R}^2\to \mathbb{R}^2$ be defined by $V_1(x,y) = (x/3, x(1-x)/2+y/2)$, $V_2(x,y) = (-x/3+1, x(1-x)/2+y/2)$ and $V_3(x,y) = (x/3+1/3, x(1-x)/2+y/2+1/2)$.  Finally, let $\mathbb{I}$ be the RIFS consisting of $\mathbb{I}_1 =\{U_1, U_2, U_3\}$ and $\mathbb{I}_2 = \{V_1, V_2\}$.  The attractor of $\mathbb{I}_1$ is a self-affine set.

\begin{figure}[H]
	\centering
	\includegraphics[width=50mm]{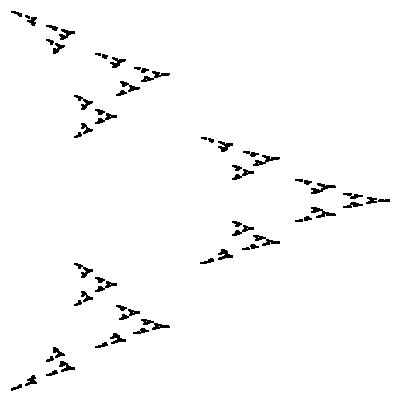}
	\qquad
	\includegraphics[width=50mm]{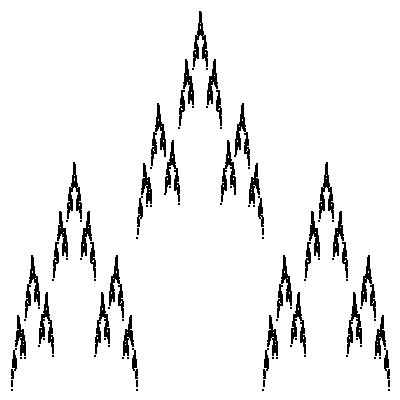} \\ \vspace{5mm}
	\includegraphics[width=50mm]{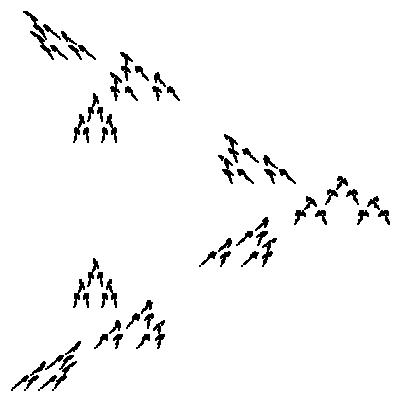}
	\qquad
	\includegraphics[width=50mm]{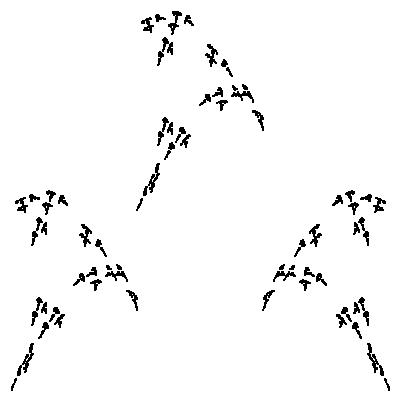}
	\caption{The attractors of $\mathbb{I}_1$ (top left) and $\mathbb{I}_2$ (top right) along with two random attractors of $\mathbb{I}$ corresponding to $\omega = (1,1,2,2,1,2,\dots)$ (bottom left) and  $\omega = (2,1,1,2,1,1,\dots)$ (bottom right).}
\end{figure}

\section{Discussion} \label{discuss}

In this section we collect together and discuss some of the questions raised by the results in this paper.
\\ \\
(1) \emph{Is the typical measure always extremal?}  We have shown that the typical dimensions behave rather well in that the typical Hausdorff and lower box dimensions always infimal and the typical Hausdorff and lower box dimensions always supremal.  The typical Hausdorff and packing measures behave rather worse and our examples show that they can both be either infimal or supremal.  However, we have not proved that they are \emph{always} extremal.
\\ \\
(2) \emph{Computing the extremal dimensions}.  Theorem \ref{main} tells us that the typical dimensions are extremal in very general circumstances.  However, it gives no indication of how one might compute the extremal dimensions.  This may be a very difficult problem and the example in Section \ref{affineexample} sheds some light on that difficulty.  Given a RIFS, can we say anything non-trivial about the extremal dimensions in general?  Theorem \ref{SS} tells us how to compute the extremal dimensions in the self-similar setting, assuming the UOSC.
\\ \\
(3) \emph{The bi-Lipschitz requirement.}  Throughout this paper we assume that all of our maps are bi-Lipschitz.  It is easily seen, however, that not all of our proofs require this.  In fact, Theorem \ref{main} parts (1), (3) and (5) go through assuming that the maps are simply contractions.  Also, a slightly weaker version of Theorem \ref{zero} can be proved, which states that if there exists $v \in \Omega$ satisfying conditon (\ref{lip+}) and $\sup_{u \in \Omega} \,\mathcal{P}^p( F_u) < \infty$, then for a typical $\omega \in \Omega$, we have $\mathcal{P}^p(F_\omega) = 0$.
\\ \\
(4) \emph{Strengthening of Theorem \ref{SS}.}  In view of the non-conformal example given in Section \ref{examples} it seems that the validity of the bounds given in Theorem \ref{SS} depend on two things: conformality; and separation properties.  It seems likely that one could prove an analogous result using conformal mappings instead of similarities and replacing each $s_i$ with the solution of Bowen's formula corresponding to the IFS, $\mathbb{I}_i$.  What could be a more interesting question is whether or not the UOSC condition is required in the self-similar case.
\\ \\
(5) \emph{Doubling gauges.}  At first sight it is somewhat curious that in Theorem \ref{main} we require that the gauge is doubling for the result concerning packing measure, but can use arbitrary gauges for Hausdorff measure.  In fact, it is not uncommon that doubling gauges play an important r\^ole when studying packing measure, see, for example, \cite{preisspacking, doublingpacking}.
\\ \\
(6) \emph{Dimension outside range}.  The example in Section \ref{affineexample} shows that the dimensions can be strictly less than the minimum of the dimensions of the attractors of the deterministic IFSs.  We have not, however, proved that the dimensions can be bigger than the maximum of the dimensions of the attractors of the deterministic IFSs
\\ \\
(7) \emph{Separation properties in the self-similar case.} In Theorems \ref{SS} and \ref{SS2} we assumed various separation properties.  In fact, some parts of these Theorems go through assuming slightly weaker conditions.  For example, Theorem \ref{SS2} (1) we require only the $\mathcal{H}^{s_{\min}}$-MSC to prove that the typical Hausdorff measure is infimal and positive and finite.  We choose to state these theorems using the stronger separation properties in order to simplify exposition and not shroud the key ideas.
\\ \\
(8) \emph{More randomness.}  It is possible to introduce more randomness into our construction.  In particular, one might relax the requirement that at the $k$th level of the construction we use the same IFS within each $k$th level iterate of $K$.  In this case our sequence space, $\Omega$, would be replaced by a space of infinite rooted trees.  We believe that although this is a significantly more general construction, the topological properties of $\Omega$ would not change significantly and most of our arguments should generalise without too much difficulty.  One might also consider the intermediate levels of randomness given by V-variable fractals introduced in \cite{vvariable} and discussed in detail in \cite{superfractals}.
\\ \\
(8) \emph{Typical versus almost sure.}  An interesting consequence of Theorem \ref{main} is that our topological approach gives drastically different results to the probabilistic (or measure theoretic) approach.  For example, compare Theorem \ref{almostsuress} with our result, Theorem \ref{SS2}.  A similar comparison has cropped up in a wide variety of situations with, roughly speaking, the topological approach favouring divergence and the probabilistic approach favouring converegence.  Indeed, our results on dimension are of this nature.  A similar phenomenon has arisen in, for example: dimensions of measures \cite{haasedimmeasure, olsendimmeasure}; dimensions of graphs of continuous functions \cite{me_hausdorff}; and frequency properties of expansions of real numbers \cite{typicalnormalnum}.  These references are given as a sample of some of the situations where a contrast between topological and probabilistic approaches have been observed and are by no means a complete list.  For example, generic dimensions of measures and graphs of continuous functions have been studied extensively and, for a more complete survey, the reader is referred to \cite{olsendimmeasure} and \cite{me_hausdorff} and the references therein.
\\ \\
(9) \emph{Choice of topological space.} Baire category theory can be used in much more general spaces than just complete metric spaces.  In fact, all one needs is a \emph{Baire topological space}, i.e., a topological space where the intersection of any countable collection of open dense sets is dense.  In Section \ref{topapp} we introduced a topology on $\Omega$ to allow us to examine the size of subsets of $\Omega$ using Baire category.  Of course we could have formulated our analysis in terms of the set $\Lambda = \{F_\omega : \omega \in \Omega\}$ equipped with the topology induced by the Hausdorff metric.  We note here that these two approaches are essentially equivalent.  Define an equivalence relation, $R$, on $\Omega$ by $\omega \,  R  \, u \Leftrightarrow F_\omega=F_u$ and let $q: \Omega \to \Omega / R$ be the quotient map, where  $\Omega / R$ is equipped with the quotient topology.  Let $\Psi: \Omega \to \mathcal{K}(K)$ be defined by $\Psi(\omega) = F_\omega$ and $\hat{\Psi}: \Omega/R \to \mathcal{K}(K)$ be defined by $\hat{\Psi}([\omega]) = F_\omega$ and observe that $\Psi$ is continuous by Lemma \ref{continuity} and that $\hat{\Psi}$ is clearly well-defined.  The following diagram commutes
\[
\begin{xy}
(0,30)*+{\Omega}="a"; (30,30)*+{\Omega/R}="b";
(30,0)*+{\Lambda}="d"; 
{\ar@{->>} "a";"b"}?*!/_2mm/{q};
{\\ \ar@{>->>} "b";"d"};?*!/_2mm/{\hat{\Psi}};
{\ar@{->>} "a";"d"};?*!/_2mm/{\Psi};
\end{xy}
\]
and furthermore, $\hat{\Psi}$ is a homeomorphism.  It is easy to see that $\Omega/R$, and hence $\Lambda$, are Baire and that images of residual subsets of $\Omega$ under $q$ are residual in $\Omega/R$.  It follows that all of our results could be phrased as `for a typical set $F_\omega \in \Lambda$...' instead of `for a typical $\omega \in \Omega$...'.
\\

\begin{centering}

\textbf{Acknowledgements}

\end{centering}

The author was supported by an EPSRC Doctoral Training Grant and thanks Kenneth Falconer for some helpful comments on a previous draft of the manuscript.

\end{document}